\documentclass[a4paper]{article}

\usepackage{amsfonts}
\usepackage{amssymb}
\usepackage{amsthm}
\usepackage{amsmath}
\usepackage{epsfig}
\usepackage{psfrag}
\usepackage{url}
\usepackage{mathdots}
\usepackage{xcolor}

\newtheorem{theorem}{Theorem}[section]
\newtheorem{corollary}[theorem]{Corollary}
\newtheorem{lemma}[theorem]{Lemma}
\newtheorem{proposition}[theorem]{Proposition}

\newtheorem{example}[theorem]{Example}
\newtheorem{remark}[theorem]{Remark}

\newtheorem{definition}[theorem]{Definition}
\newtheorem{conjecture}[theorem]{Conjecture}
\newtheorem{question}[theorem]{Question}

\newtheorem{*theorem}[theorem]{*Theorem}
\newtheorem{*corollary}[theorem]{*Corollary}
\newtheorem{*lemma}[theorem]{*Lemma}
\newtheorem{*proposition}[theorem]{*Proposition}

\def\N{\mathbb{N}}

\def\real{\mathbb{R}}

\def\C{\mathbb{C}}

\def\normal{\mathcal{N}}
\def\gauss{\mathcal{G}}

\def\conv{\operatorname{conv}}
\def\cone{\operatorname{cone}}

\def\vertices{\operatorname{vertices}}
\def\d{\operatorname{dist}}

\def\ops{\operatorname{S}}

\begin{document}

\title{A counterexample to the Hirsch conjecture}

\author{
Francisco Santos\thanks{
Supported in part by the Spanish Ministry of Science through grants  MTM2008- 04699-C03-02 and CSD2006-00032 (i-MATH) and by MICINN-ESF EUROCORES programme EuroGIGA - ComPoSe IP04 - Project EUI-EURC-2011-4306.
}
}

\date{To Victor L.~Klee (1925--2007), in memoriam\footnote{I met Vic Klee only once, 
when I visited the University of Washington to give a seminar talk.
Although he was retired---he was already 76 years old---he came 
to the Department of Mathematics to talk to me and we
had a nice conversation during which he asked me 
``Why don't you try to disprove the Hirsch Conjecture?'' 
This work is the answer to that question.
}}

\maketitle

\begin{abstract}
The Hirsch Conjecture (1957) stated that the graph of a $d$-dimensional polytope with $n$ facets 
cannot have (combinatorial) diameter greater than $n-d$. That is, that any two vertices of the polytope can be connected by a path of at most $n-d$ edges. 

This paper presents the first counterexample to the conjecture.
Our polytope has dimension $43$ and $86$ facets. It is obtained from a $5$-dimen\-sio\-nal polytope with $48$ facets  
which violates a certain generalization of the $d$-step conjecture of Klee and Walkup.
\end{abstract}

\section{Introduction}

The Hirsch conjecture is the following fundamental statement about the combinatorics of polytopes.
It was stated by Warren M.~Hirsch in 1957 in the context of the \emph{simplex method},
and  publicized by G.~Dantzig in his 1963 monograph on linear programming~\cite{Dantzig-book}:

\begin{quote}
The (combinatorial) diameter of a polytope  of dimension $d$ with $n$ facets  cannot be greater than $n-d$.
\end{quote}

Here we call \emph{combinatorial diameter} of a polytope the maximum number of steps needed to go from one vertex to another, where a step consists in traversing an edge. Since we never refer to any other diameter in this paper, we will often omit the word ``combinatorial''. We say that a polytope is \emph{Hirsch} if it satisfies the conjecture, and \emph{non-Hirsch} if it does not.

Our main result (Corollary~\ref{coro:nonHirsch}) is the construction of a $43$-dimensional polytope with $86$ facets and diameter (at least) $44$. Via products and glueing copies of it we can also construct an infinite family of polytopes in fixed dimension $d$ with increasing number $n$ of facets and of diameter bigger than $(1+\epsilon) n$, for a positive constant $\epsilon$ (Theorem~\ref{thm:asymptotic}).

\subsection*{Linear programming, the simplex method and the Hirsch Conjecture}

Linear programming (LP)
is the problem of maximizing (or minimizing) a linear functional subject to linear inequality constraints.
%
%
For more than 30 years the only applicable method for LP was
the \emph{simplex method}, devised in 1947 by G.~Dantzig~\cite{Dantzig-simplex}. This method solves a linear program by first finding a vertex of the \emph{feasibility region} $P$, which is a facet-defined \emph{polyhedron}, and then
jumping from vertex to neighboring vertex along the edges of $P$, always increasing the functional to be maximized.  When such a \emph{pivot step} can no longer increase the functional, convexity guarantees that we are at the global maximum. (A \emph{pivot rule} has to be specified for the algorithm to choose among the possible neighboring vertices; the performance of the algorithm may depend on the choice).

To this day, the complexity of the simplex method is quite a mystery:  exponential (or almost) worst-case behavior of the method is known for most of the pivot rules practically used or theoretically proposed. (For two recent breakthrough additions see~\cite{Friedman, Friedman-Hansen-Zwick}). But on the other hand, as M.~Todd recently put it, \emph{``the number of steps [that the simplex method takes] to solve a problem with $m$ equality constraints in $n$ nonnegative variables is almost always at most a small multiple of $m$, say $3m$''}~\cite{Todd:Dantzig}. Because of this ``\emph{the simplex method has remained, if not the method 
of choice, a method of choice, usually competitive with, and on some classes of problems superior to, the more modern approaches}''~\cite{Todd:Dantzig}. This is so even after the discovery, 30 years ago, of polynomial time algorithms for linear programming by Khachiyan and Karmarkar~\cite{Karmarkar, Khachiyan}.
In fact, in the year 2000 the simplex method was selected as one of the ``10 algorithms with the greatest influence on the development and practice of science and engineering in the 20th century'' by the journal \emph{Computing in Science and Engineering}~\cite{Dongarra-Sullivan-topten}.

It is also worth mentioning that
Khachiyan and Karmarkar's algorithms are polynomial in the \emph{bit model} of complexity but they are not polynomial  in the  \emph{real number machine} model of Blum et al.~\cite{BCSS,BSS}. Algorithms that are polynomial in both models are called \emph{strongly polynomial}. S.~Smale~\cite{Smale} listed among his ``mathematical problems for the next century'' the question whether linear programming can be performed in strongly polynomial time. A polynomial pivot rule for the simplex method would answer this in the affirmative. 

%
For information on algorithms for linear programming and their complexity, including attempts to explain the good behavior of the simplex method \emph{without} relying on bounds for the diameters of \emph{all} polytopes, 
see~\cite{Borgwardt-book, GritzmannKlee, Megiddo:SurveyLPComplexity, Spielman:WhySimplexUsually,Terlaky-Zhang}.


\subsection*{Brief history of the Hirsch conjecture}

Warren M.~Hirsch (1918--2007), a professor of probability at the Courant Institute, communicated his conjecture to G.~Dantzig in connection to the simplex method: the diameter of a polytope is a lower bound (and, hopefully, an approximation) to the number of steps taken by the simplex method in the worst-case. Hirsch had verified the conjecture for $n-d\le 4$ and Dantzig included this statement and the conjecture in his 1963 book~\cite[p.~160]{Dantzig-book}.

The original conjecture did not distinguish between \emph{bounded} or \emph{unbounded} feasibility regions. In modern terminology, a bounded one is a \emph{(convex) polytope} while a perhaps-unbounded one is a \emph{polyhedron}. But the unbounded case was disproved by Klee and Walkup~\cite{Klee:d-step} in 1967 with the construction of a polyhedron of dimension $4$ with $8$ facets and diameter $5$. Since then the expression ``Hirsch Conjecture'' has been used referring to the bounded case.

In the same paper, Klee and Walkup established the following crucial statement. See a proof in Section~\ref{sec:general-d-step}:

\begin{lemma}[Klee, Walkup~\cite{Klee:d-step}]
\label{lemma:dstep}
For positive integers $n>d$, 
let $H(n,d)$ denote the maximum possible diameter of the graph of a $d$-polytope with $n$ facets. Then, $H(n,d)\le H(2n-2d,n-d)$. Put differently:
\[
\forall m\in\N, \qquad \max_{d\in \N} \{H(d+m,d)\} = H(2m,m).
\] 
\end{lemma}

\begin{corollary}[{\emph{$d$-step Theorem}, Klee-Walkup~\cite[Theorem~2.5]{Klee:d-step}}]
\label{coro:dstep}
The following statements are equivalent:
\begin{enumerate}
\item $H(n,d)\le n-d$ for all $n$ and $d$. \hfill \emph{(Hirsch Conjecture)}
\item $H(2d,d)\le d$ for all $d$. \hfill \emph{($d$-step Conjecture)}
\end{enumerate}
\end{corollary}

The Hirsch conjecture holds for $n\le d+6$: Klee and Walkup proved the $d$-step conjecture for $d\le 5$, and the case $d=6$ has recently been verified by Bremner and Schewe~\cite{Bremner:DiameterFewFacets}. In a previous paper~\cite{Klee:PathsII}, Klee had shown the Hirsch conjecture for $d=3$. Together with the cases $(n,d)\in\{(11,4), (12,4)\}$~\cite{Bremner:MoreBounds, Bremner:DiameterFewFacets}, these are all the parameters where the Hirsch conjecture is known to hold.

The best upper bounds we have for $H(n,d)$ in general are a \emph{quasi-polynomial} one by Kalai and Kleitman~\cite{Kalai:Quasi-polynomial} and one \emph{linear in fixed dimension} proved by Barnette and improved by Larman and then Barnette again~\cite{Barnette, Barnette2, Larman}. These bounds take the following form (the second one assumes $d\ge 3$):
\[
H(n,d)\le n^{\log_2(d) + 1},
\qquad
H(n,d)\le \frac{2^{d-2}}{3}n.
\]

In particular, no polynomial upper bound is known. Its existence is dubbed the \emph{polynomial Hirsch conjecture}:

\begin{conjecture}[Polynomial Hirsch Conjecture]
\label{conj:pyolyhirsch}
There is a polynomial $f(n)$ such that the diameter of every polytope with $n$ facets is bounded above by $f(n)$.
\end{conjecture}


%
Of course, apart of its central role in polytope theory, the significance of this conjecture is that 
polynomial pivot rules cannot exist for the simplex method unless it holds.
For more information on these and other results see the chapter that Klee wrote in Gr\"unbaum's book~\cite{Klee-grunbaumchapter}, and the survey papers~\cite{Kim-Santos-update, Klee-Kleinschmidt}.

\subsection*{Our counterexample}
The $d$-step Theorem stated above implies that to prove or disprove the Hirsch conjecture there is no loss of generality in assuming $n=2d$. A second reduction is that, for every $n$ and $d$, the maximum of $H(n,d)$ is always achieved at a \emph{simple} polytope (a $d$-polytope in which every vertex belongs to exactly $d$ facets). Simple polytopes are especially relevant for linear programming: they are obtained when the system of inequalities is sufficiently generic.

The first ingredient in our proof is the observation that if we start with a polytope which is at the same time non-simple and has $n>2d$, then the techniques used in these two reductions can be combined to get a simple $(n-d)$-polytope with $2n-2d$ facets and not only maintain its diameter (or, to be more precise, the distance between two distinguished non-simple vertices), but actually \emph{increase it}. We call this the \emph{Strong $d$-step Theorem} and prove it in Section~\ref{sec:general-d-step}. Although the theorem can be stated more generally (see Remark~\ref{rem:strong-d-step}), the version we need has to do with the following class of polytopes:

\begin{definition}
\label{defi:spindle}
A $d$-\emph{spindle} is a $d$-polytope $P$ having two distinguished vertices $u$ and $v$ such that every facet of $P$ contains exactly one of them. See Figure~\ref{fig:spindle}. The \emph{length} of a spindle is the graph distance between $u$ and $v$.
\end{definition}

Equivalently, a spindle is  the intersection of two polyhedral convex cones with apices at $u$ and $v$ and with both their interiors containing the open segment $uv$.

\begin{figure}[htb]
\begin{center}
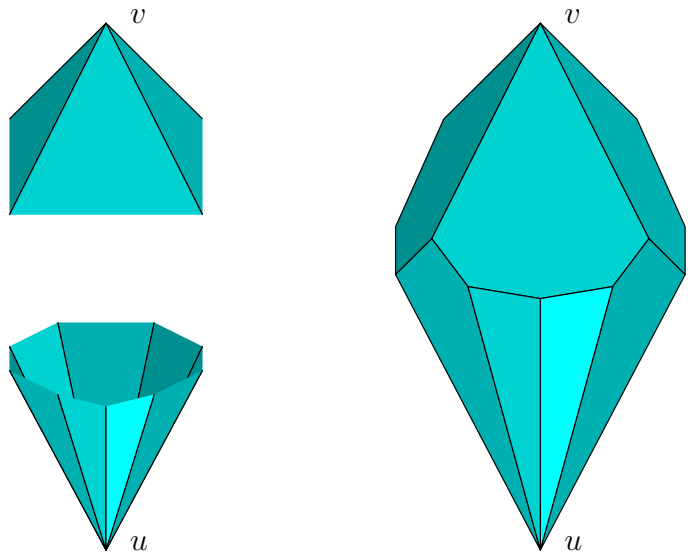
\caption{A spindle. 
}
\label{fig:spindle}
\end{center}
\end{figure}

\begin{theorem}[Strong $d$-step Theorem for spindles]
\label{thm:dstep-spindle}
If $P$ is a spindle of dimension $d$, with $n$ facets and length $l$, then there is another spindle $P'$ of dimension $n-d$, with $2n-2d$ facets and with length at least $l+n-2d$.
In particular, if $l>d$ then $P'$ violates the $d$-step conjecture, hence also the Hirsch conjecture.
\end{theorem}


The second ingredient in our disproof is the explicit construction of a spindle of dimension five and length six, that we describe in Section~\ref{sec:symmetry}:

\begin{theorem}
\label{thm:spindle}
There is a $5$-dimensional spindle (with $48$ facets and $322$ vertices) of length six.
\end{theorem}

That our spindle has length six is easy to verify computationally.
Still, we include two computer-free proofs in Sections~\ref{sec:prismatoid} and~\ref{sec:sphere-maps}. 
Putting Theorems~\ref{thm:dstep-spindle} and~\ref{thm:spindle} together we get:

\begin{corollary}
\label{coro:nonHirsch}
There is a non-Hirsch polytope of dimension $43$ with $86$ facets.
\end{corollary}

Section~\ref{sec:asymptotic} is devoted to show how to derive an infinite family of non-Hirsch polytopes from the first one:

\begin{theorem}
\label{thm:asymptotic}
There is a fixed dimension $d$, a positive $\epsilon>0$, and an infinite family of $d$-polytopes $P_k$ each with $n_k$ facets and with diameter bigger than $(1+\epsilon)n_k$.
\end{theorem}

For example, from the non-Hirsch polytope in this paper we get  $\epsilon \simeq1/86$ in dimension $86$ and  $\epsilon \simeq 1/43$ in high $d$. 
With the one announced in~\cite{5prismatoids} (see Theorem~\ref{thm:5prismatoids} below) one gets  $\epsilon \simeq1/40$ in dimension $40$ and  $\epsilon \simeq 1/20$ in high $d$. 

\subsection*{Discussion}

Our counterexample disproves as a by-product the following two statements, 
originally posed in the hope of shedding light on the Hirsch conjecture:
\begin{enumerate}
\item Provan and Billera~\cite{ProvanBillera} introduced the hierarchy of $k$-decomposable simplicial complexes: $k$-decomposability is stronger than $(k+1)$-decomposability for every $k$ and the boundary of every simplicial $d$-polytope is $(d-1)$-decomposable (or \emph{shellable}). 
They also showed that $0$-decomposable (or \emph{vertex decomposable}) complexes satisfy (the polar of)  the Hirsch conjecture. Non-vertex-decomposable polytopes were previously found by Kleinsch\-midt~\cite[p.~742]{Klee-Kleinschmidt}, but it would be interesting to explore
whether our non-Hirsch polytope, besides not being $0$-decomposable, fails also to be $1$-decomposable (or higher).

\item Todd~\cite{Todd:MonotonicBoundedHirsch} showed that from any unbounded non-Hirsch polyhedron, such as the one previously found by Klee and Walkup, one can easily obtain a counterexample to the so-called \emph{monotone Hirsch conjecture}. Still, the \emph{strict monotone Hirsch conjecture} of Ziegler~\cite{Ziegler:LecturesPolytopes}, stronger than the Hirsch conjecture, was open.
\end{enumerate}

Still, our techniques leave the underlying problem--how large can the diameter of a polytope be--almost as open as it was before. In particular, we cannot answer the following question:

\begin{question}
Is there a constant $c$ (independent of $d$) such that the diameter of every $d$-polytope with $n$ facets is bounded above by $cn$?
\end{question}

We suspect the answer to be negative, but our lack of knowledge somehow confirms the following sentence from~\cite{Klee-Kleinschmidt}:
\emph{Finding a counterexample will be merely a small first step in the line of investigation related to the conjecture.}

Another ``small step''  has recently been given by F.~Eisenbrand, N.~H\"ahnle, A.~Razborov, and T.~Rothvo\ss~\cite{Eisenbrand:limits-of-abstraction}, with the introduction of certain abstract generalizations of boundary complexes of polyhedra and the construction of objects of super-linear diameter in this generalized setting. By further analyzing this setting, N.~H\"ahnle has posed the following tempting and more explicit version of Conjecture~\ref{conj:pyolyhirsch}:

\begin{conjecture}[\cite{Hahnle:polymath3}]
\label{conj:haehnle}
The diameter of every $d$-polytope with $n$ facets is bounded above by $dn$.
\end{conjecture}

Thanks to Remark~\ref{rem:mcmullen} this conjecture is (almost) equivalent to:

\begin{conjecture}
\label{conj:haehnle-2}
The diameter of every $d$-polytope with $n$ facets is bounded above by $d(n-d)$.
\end{conjecture}

To finish, let us mention two additional results that were obtained after the first version of this paper was made public:
%
%
On the one hand, relying on the reductions that we introduce in Section~\ref{sec:maps},
Santos, Stephen and Thomas~\cite{4prismatoids} have shown that all $4$-spindles have length at most $4$~\cite{4prismatoids}. Hence, spindles of dimension five are truly needed to obtain non-Hirsch polytopes via Theorem~\ref{thm:dstep-spindle}.
%
On the other hand, Matschke, Santos and Weibel~\cite{5prismatoids} have constructed a $5$-spindle of length $6$ with only $25$ facets, from which Theorem~\ref{thm:spindle} gives:

\begin{theorem}[Matschke, Santos, Weibel~\cite{5prismatoids}]
\label{thm:5prismatoids}
There is a non-Hirsch polytope of dimension $20$ with $40$ facets and $36\,442$ vertices. It has diameter $21$.
\end{theorem}

Apart of the decrease in dimension, the smaller size of this example has allowed us to explicitly compute coordinates for the non-Hirsch polytope in question. Doing the same with the $43$-dimensional example presented in this paper seemed out of reach. We would need to apply $38$ times the operation of \emph{wedge followed by perturbation} in the proof of Theorem~\ref{thm:dstep-spindle} (see Section~\ref{sub:strong-dstep}). Julian Pfeifle (personal communication) wrote a small program to automatically do this and, using a standard desktop computer with 2GB of RAM, was able to undertake the first nine iterations. Experimentally, he found that each iteration more or less doubled the number of vertices (and multiplied by four or five the computation time) indicating that the final non-Hirsch polytope has about $2^{40}$ vertices. To make things worse, the tower of $38$ perturbations would give rise to either huge rational coefficients or very delicate numerical approximation issues.


%



\section{A strong $d$-step theorem for spindles}
\label{sec:general-d-step}

We find it easier to work in a polar setting in which we want to travel from facet to facet of a polytope $Q$ crossing \emph{ridges} (codimension-$2$ faces), rather than travel from vertex to vertex along edges.
That is, we are interested in the following dual version of the Hirsch conjecture:

\begin{definition}
A $d$-polytope $Q$ with $n$ vertices is a  \emph{dual-Hirsch} polytope if $n-d$ dual steps suffice to travel from any facet of $Q$ to any other facet. A \emph{dual step} consists in moving from one facet $F$ of $Q$ to an adjacent one $F'$, meaning by this that $F$ and $F'$ share a ridge of $Q$.
%
\end{definition}

Clearly, $Q$ is dual-Hirsch if and only if its polar polytope is Hirsch. In the rest of the paper we omit the word dual from our dual paths and dual steps.

\subsection{Two classical reductions}
It is known since the 60's that to prove or disprove the (dual) Hirsch conjecture it is enough to look at \emph{simplicial polytopes} with \emph{twice as many vertices as their dimension}. Since our strong $d$-step theorem is based in combining both reductions, let us see how they work.

For the first reduction, following Klee~\cite{Klee:diameters, Klee:number-of-vertices} we use the operation of \emph{pushing} vertices. Let $Q$ be a polytope with vertices $V$ and let $v\in V$ be one of them.
We say that a polytope $Q'$ is obtained from $Q$ by \emph{pushing $v$ }Êif the vertices of $Q'$ are $V\setminus\{v\}\cup \{v'\}$ for a certain point $v'\in Q$ and the only hyperplanes spanned by vertices of $Q$ that intersect the segment $vv'$ are those containing $v$. 
Put differently, the vertex $v$ is \emph{pushed} to a new position $v'$ within the polytope $Q$ but sufficiently close to its original position. We emphasize that we admit $v'$ to be in the boundary of $Q$. In the standard notion of pushing, $v'$ is required to be in the interior of $Q$.

\begin{lemma}
\label{lemma:pushing}
Let $Q'$ be obtained from $Q$ by pushing $v$. Then:
\begin{enumerate}
\item Let $F'$ be a facet of $Q'$ with vertex set $S'$ and let $S=S'\setminus\{v'\}\cup\{v\}$ or $S=S'$ depending on whether $v' \in F'$ or not. Then, there is a unique facet $\phi(F')$ of $Q$ such that $S\subset \phi(F')$.
\item The map $F'\mapsto \phi(F')$ sends adjacent facets of $Q'$ to either the same or adjacent facets of $Q$. (That is, $\phi$ is a simplicial map between the dual graphs of $Q'$ and $Q$).
\end{enumerate}
\end{lemma}

\begin{proof}
For part (1), consider what happens when we continuously move $v'$ to its original position $v$ along the segment $vv'$. We call $Q(t)$, $F(t)$, $S(t)$ and $v(t)$ the polytope, facet, vertex set, and vertex obtained at moment $t$, with $v'=v(1)$ and $v=v(0)$.
The assumption that no hyperplane spanned by vertices of $Q$ intersects $vv'$ unless it contains $v$ implies that  the combinatorics of $Q(t)$ remains the same at every moment $t>0$; changes will happen only at $t= 0$. Now, every facet-defining hyperplane of $Q(t)$ will tend to a facet-defining hyperplane of $Q(0)$ (which implies part (1)) unless the vertex set $S(t)$ spanning a certain facet $F(t)$ collapses to lie in a  flat of codimension two. Put differently, unless $F(t)$ is a pyramid with apex at $v(t)$ over a codimension two face $G'$ of $Q'$ with $v$ in the affine span of $G'$.
We claim that the assumption $v'\in Q$ rules out this possibility. Indeed, in this situation the hyperplane $H'$ spanned by $F(t)$ is independent of $t$ for $t>0$ and it contains the segment $vv'$. Let $w$ be the last point where the ray from $v$ through $v'$ meets $Q$. Then, $w$ is a convex combination of vertices of $Q$ different from $v$ and it lies in the hyperplane $H'$, so it lies in the facet $F(t)$ for every $t>0$. Since $w$ is further from $G'$ than $v(t)$, $F(t)$ cannot be a pyramid with apex at $v(t)$ and base $G'$.

For part (2) we reinterpret (the proof of) part (1) as saying: when $t$ goes from $1$ to $0$ the combinatorics of $Q'$ remains the same except that at $t=0$ some groups of facets of $Q'$ merge to single facets of $Q$. This implies the claim.
\end{proof}

\begin{corollary}[Klee~\cite{Klee:diameters}]
For every polytope $Q$  there is a simplicial polytope $Q'$  of the same dimension and number of vertices  and with the same or greater dual diameter.
\qed
\end{corollary}

For the next lemma, and for the proof of Theorem~\ref{thm:dstep-prismatoid} we introduce the \emph{one-point-suspension}.
For a given vertex $v$ of a $d$-polytope $Q$ with $n$ vertices, the one-point-suspension
$\ops_v(Q)$ is constructed embedding $Q$ in a hyperplane in $\real^{d+1}$ and adding two new vertices $u$ and $w$ on opposite sides of that hyperplane and such that the segment $uw$ contains $v$. This makes $\ops_v(Q)$ have dimension $d+1$ and $n+1$ vertices (we added two, but $v$ is no longer a vertex). 
The facets of $\ops_v(Q)$ are of two types:
\begin{eqnarray*}
\ops_v(F), & & \text{ for each facet $F$ of $Q$ with $v\in F$, and} \\
F*u \text{ and } F*w,& & \text{ for each facet $F$ of $Q$ with $v\not\in F$.}  
\end{eqnarray*}
In this formula, $F*u$ denotes the pyramid over $F$ with apex at $u$.
See~\cite{triang-book} or \cite{Kim-Santos-companion}  for more details, and Figure~\ref{fig:ops} for an illustration. One-point-suspensions appear in the literature also under other names, such as \emph{dual wedges} or \emph{vertex splittings}. 

%
\begin{figure}[htb]
\begin{center}
\input{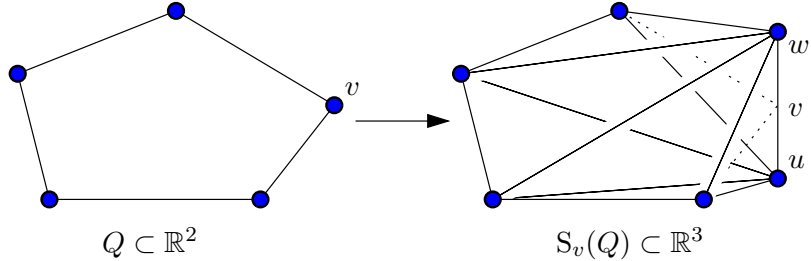}
\caption{The one-point-suspension of a pentagon is a simplicial $3$-polytope with six vertices and eight facets}
\label{fig:ops}
\end{center}
\end{figure}

\begin{lemma}
\label{lemma:ops}
Let $F_1$ and $F_2$ be two facets of a polytope $Q$ and let $\widetilde Q:=\ops_v(Q)$ for some vertex $v$. For $i=1,2$, let $\widetilde F_i\le \widetilde Q$ denote the facet $\ops_v(F_i)$ if $v\in F_i$ or one of the facets $F_i*u$ or $F_i*w$ if $v\not\in F_i$. Then, the  distance between $\widetilde F_1$ and $\widetilde F_2$ in the dual graph of $\widetilde Q$ is greater or equal than  the  distance between $F_1$ and $F_2$ in the dual graph of $Q$.
\end{lemma}

\begin{proof}
The  dual graph of $\widetilde Q$ projects down to that of $Q$ by sending each facet $\ops_v(F)$, $F*u$ or $F*w$ to the facet $F$ of $Q$ that it came from. Graph-theoretically, this projection amounts to contracting all the dual edges between $F*u$ and $F*w$, for each facet $F\le Q$ not containing $v$.
\end{proof}


\begin{proof}[Proof of Lemma~\ref{lemma:dstep}]
We prove that $H(n,d) \le H(2n-2d, n-d)$ by induction on $|n-d|$ and separating the cases $n>2d$ and $n<2d$:
If $n<2d$ then every pair $u,v$ of vertices of a $d$-polytope $P$ with $n$ facets lie in some common facet $F$. $F$ is a polytope of dimension $d-1$ with at most $n-1$ facets so the distance from $u$ to $v$ in  $F$ is bounded by $H(n-1,d-1)$.
%
If $n>2d$, apply Lemma~\ref{lemma:ops} to a $d$-polytope $Q$ with $n$ vertices whose dual diameter achieves $H(n,d)$, to get $H(n,d)\le H(n+1,d+1)$.
\end{proof}

\subsection{The strong $d$-step Theorem}
\label{sub:strong-dstep}

The following class of polytopes are the polars of the \emph{spindles} mentioned in Theorem~\ref{thm:dstep-spindle}:

\begin{definition}
A \emph{prismatoid} is a polytope having two parallel facets $Q^+$ and $Q^-$ that contain all vertices. We call $Q^+$ and $Q^-$ the \emph{base facets} of $Q$.
The \emph{width} of a prismatoid is the dual graph distance between $Q^+$ and $Q^-$.
\end{definition}


The ``base facets'' of a prismatoid may not be unique, but they are part of the definition. For example, a cube or an octahedron are prismatoids with respect to any of their pairs of opposite facets. 
Observe that the requirement of $Q^+$ and $Q^-$ to be parallel is not especially relevant, as long as they are disjoint. If $Q^+$ and $Q^-$ are disjoint facets of an arbitrary polytope $Q$ then a projective transformation can make them parallel without changing the combinatorics of $Q$ (see, e.~g., \cite[p.~69]{Ziegler:LecturesPolytopes}).
%
%
The following statement is equivalent to Theorem~\ref{thm:dstep-spindle}:

\begin{theorem}[Strong $d$-step Theorem for prismatoids]
\label{thm:dstep-prismatoid}
If $Q$ is a prismatoid of dimension $d$ with $n$ vertices and width $l$, then there is another prismatoid $Q'$ of dimension $n-d$, with $2n-2d$ vertices and width at least $l+n-2d$.
In particular, if $l>d$ then $Q'$ violates the (dual) $d$-step conjecture, hence also the (dual) Hirsch conjecture.
\end{theorem}

\begin{proof}
%
We call the number $s=n-2d$ the \emph{asimpliciality} of $Q$ and
prove the theorem by induction on $s$. Since every facet of a $d$-polytope has at least $d$ vertices, $s$ is always non-negative. The  base case of $s=0$ is tautological. For the inductive step, we show that if $s>0$ we can construct from $Q$ a new prismatoid $\widetilde Q$ 
with dimension one higher, one vertex more (in particular, the ``asimpliciality'' has decreased by one), and width \emph{at least} one more than $Q$.

Since $s>0$, 
at least one of $Q^+$ and $Q^-$, say $Q^+$, is not a simplex (the other one, $Q^-$, may or may not be a simplex). Let $v$ be  a vertex of $Q^-$, and let $\ops_v(Q)$ be the one-point-suspension of $Q$ over $v$. 
%
Let $\widetilde Q^- = \ops_v(Q^-)$ be the one-point-suspension of $Q^-$, which appears as a facet of $\ops_v(Q)$. Observe that 
$\ops_v(Q)$ is almost a prismatoid: its faces $\widetilde Q^-$ and $Q^+$ contain all vertices and they lie in two parallel hyperplanes. 
The only problem is that $Q^+$ is not a facet, it is a ridge; but since we know that $Q^+$ is not a simplex, moving its vertices slightly in the direction of  the segment $uw$ creates a new facet $\widetilde Q^+$ parallel to $\widetilde Q^-$, and  these two facets contain all the vertices of the new polytope, that we denote $\widetilde{Q}$ and is a prismatoid. (In fact, moving a single vertex of $Q^+$ is enough to achieve this). See Figure~\ref{fig:ops-prismatoid} for an illustration. In the figure, we draw several points along the edge $Q^+$ to convey the fact that $Q^+$ is not a simplex; these points have to be understood as vertices of $Q$.

\begin{figure}[htb]
\begin{center}
\input{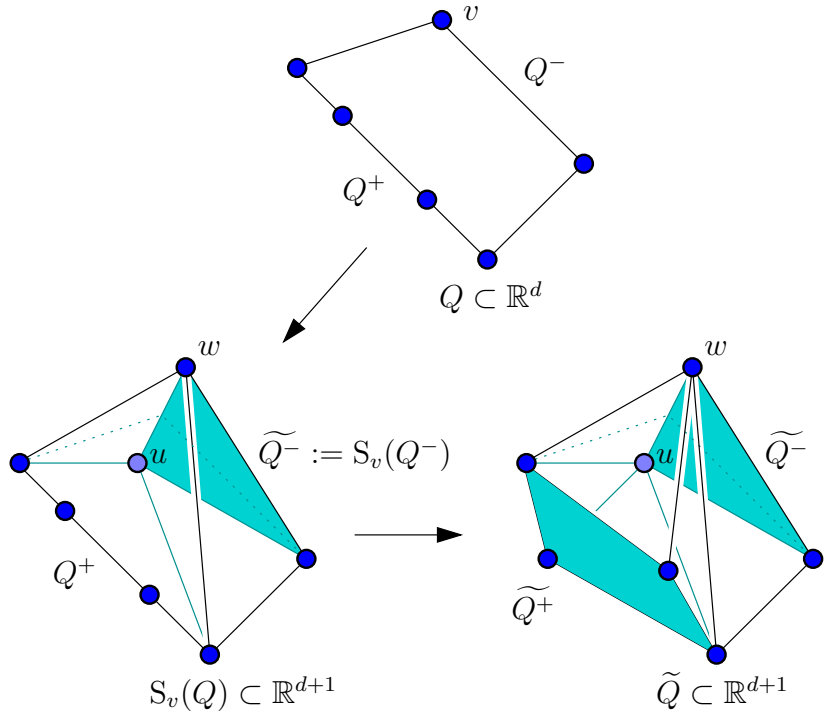}
\caption{Perturbing a base facet in the one-point-suspension $\ops_v(Q)$ of a prismatoid we get a new prismatoid $\widetilde{Q}$. The extra dots in the facet $Q^+$ are meant to be vertices of it, and convey the hypothesis that $Q^+$ is not a simplex
}
\label{fig:ops-prismatoid}
\end{center}
\end{figure}

By Lemma~\ref{lemma:ops}, the distance from $\widetilde Q^-$ to any of $Q^+*u$ and $Q^+*w$ in $\ops_v(Q)$ is (at least) $l$. 
To make sure that the width of $\widetilde Q$ is at least $l+1$ we do the perturbation from $\ops_v(Q)$ to $\widetilde Q$ in the following special manner: let $a$ be a vertex of $Q^+$ and assume that the only non-simplicial facets of $\ops_v(Q)$ containing $a$ are $Q^+*u$ and $Q^+*w$ (if that is not the case, we first push $a$ to a point  in the interior of $Q^+$ but otherwise generic, which maintains all the properties we need and does not decrease the dual distances, by Lemma~\ref{lemma:pushing}). We now get $\widetilde Q$ by moving only  $a$, in a direction parallel to $Q^-$ and away from $Q^+$, to a position $a'$. Then $Q^+$ will be replaced by a pyramid $\widetilde Q^+ := \conv(\vertices(Q^+)\setminus a)*a'$.

The genericity assumption on $a$ implies that 
apart from the creation of $\widetilde Q^+$, the only change to the face lattice of $\ops_v(Q)$ is that the two facets $Q^+*u$ and $Q^+*w$ get refined
as two complexes $U*u$ and $W*w$ where $U$ and $W$ are the lower and upper envelopes of the facet $\widetilde Q^+$ (here we are considering the direction $uw$ as vertical, with $w$ above $u$).
%
%
The width of $\widetilde Q$ is at least $l+1$ since in order to get out of $\widetilde Q^+$ the first step will send us to a facet in either $U*u$ or $W*w$, and that facet is at least at the same distance from $\widetilde Q^-$ as $Q^+*u$ or $Q^+*w$ were, by the same arguments as in the proof of Lemma~\ref{lemma:pushing}.
\end{proof}


\begin{remark}
\label{rem:strong-d-step}
\rm

The following version of the strong $d$-step theorem is also true, with essentially the same proof: Let $Q$ be a $d$-polytope with $n$ vertices and containing two disjoint facets $Q^+$ and $Q^-$ that use in total $m$ of the vertices, for some $m > 2d$. Let $l$ be the dual graph distance from $Q^+$ to $Q^-$. Then, there is a $(m-d)$-polytope $Q'$ with $n+m-2d$ vertices and having two facets ${Q^+}'$ and ${Q^-}'$ at distance $l+m-2d$. In particular, if $l> (n-m)+d$ then $Q'$ is non-Hirsch. 
The prismatoid version is the case $n=m$.
\end{remark}


\section{A $5$-prismatoid without the $d$-step property}
\label{sec:symmetry}

In the light of Theorem~\ref{thm:dstep-prismatoid}, we say that a prismatoid has the $d$-step property if its width does not exceed its dimension. 
It is an easy exercise to show that every $3$-dimensional prismatoid has this property. For $4$-dimensional ones, the result is still true is true, although not obvious anymore~\cite{4prismatoids}. In dimension five, however, we have the following statement, which implies  Theorem~\ref{thm:spindle}:

\begin{theorem}
\label{thm:prismatoid}
The $5$-dimensional prismatoid with  the $48$ rows of the matrices of Table~\ref{table:prismatoid} as vertices
has width six.
\end{theorem}

\begin{table}[htb]
\small
\[
\bordermatrix{
&x_1&x_2&x_3&x_4&x_5\cr
\mathit{1}^+&  18&   0&   0&   0&   1 \cr
\mathit{2}^+& -18&   0&   0&   0&   1 \cr
\mathit{3}^+&   0&  18&   0&   0&   1 \cr
\mathit{4}^+&   0& -18&   0&   0&   1 \cr
\mathit{5}^+&   0&   0&  45&   0&   1 \cr
\mathit{6}^+&   0&   0& -45&   0&   1 \cr
\mathit{7}^+&   0&   0&   0&  45&   1 \cr
\mathit{8}^+&   0&   0&   0& -45&   1 \cr
\mathit{9}^+&  15&  15&   0&   0&   1 \cr 
\mathit{10}^+& -15&  15&   0&   0&   1 \cr
\mathit{11}^+&  15& -15&   0&   0&   1 \cr
\mathit{12}^+& -15& -15&   0&   0&   1 \cr
\mathit{13}^+&   0&   0&  30&  30&   1 \cr
\mathit{14}^+&   0&   0& -30&  30&   1 \cr
\mathit{15}^+&   0&   0&  30& -30&   1 \cr
\mathit{16}^+&   0&   0& -30& -30&   1 \cr
\mathit{17}^+&   0&  10&  40&   0&   1 \cr
\mathit{18}^+&   0& -10&  40&   0&   1 \cr
\mathit{19}^+&   0&  10& -40&   0&   1 \cr
\mathit{20}^+&   0& -10& -40&   0&   1 \cr
\mathit{21}^+&  10&   0&   0&  40&   1 \cr
\mathit{22}^+& -10&   0&   0&  40&   1 \cr
\mathit{23}^+&  10&   0&   0& -40&   1 \cr
\mathit{24}^+& -10&   0&   0& -40&   1 \cr
}
\qquad\quad
\bordermatrix{
&x_1&x_2&x_3&x_4&x_5\cr
\mathit{1}^-&   0&   0&   0&  18&  -1 \cr
\mathit{2}^-&   0&   0&   0& -18&  -1 \cr
\mathit{3}^-&   0&   0&  18&   0&  -1 \cr
\mathit{4}^-&   0&   0& -18&   0&  -1 \cr
\mathit{5}^-&  45&   0&   0&   0&  -1 \cr
\mathit{6}^-& -45&   0&   0&   0&  -1 \cr
\mathit{7}^-&   0&  45&   0&   0&  -1 \cr
\mathit{8}^-&   0& -45&   0&   0&  -1 \cr
\mathit{9}^-&   0&   0&  15&  15&  -1 \cr
\mathit{10}^-&   0&   0&  15& -15&  -1 \cr
\mathit{11}^-&   0&   0& -15&  15&  -1 \cr
\mathit{12}^-&   0&   0& -15& -15&  -1 \cr
\mathit{13}^-&  30&  30&   0&   0&  -1 \cr
\mathit{14}^-& -30&  30&   0&   0&  -1 \cr
\mathit{15}^-&  30& -30&   0&   0&  -1 \cr
\mathit{16}^-& -30& -30&   0&   0&  -1 \cr
\mathit{17}^-&  40&   0&  10&   0&  -1 \cr
\mathit{18}^-&  40&   0& -10&   0&  -1 \cr
\mathit{19}^-& -40&   0&  10&   0&  -1 \cr
\mathit{20}^-& -40&   0& -10&   0&  -1 \cr
\mathit{21}^-&   0&  40&   0&  10&  -1 \cr
\mathit{22}^-&   0&  40&   0& -10&  -1 \cr
\mathit{23}^-&   0& -40&   0&  10&  -1 \cr
\mathit{24}^-&   0& -40&   0& -10&  -1 \cr
}
\]
\caption{The $48$ vertices of a $5$-prismatoid without the $d$-step property}
\label{table:prismatoid}
\end{table}

$Q$ is small enough for the statement of Theorem~\ref{thm:prismatoid} to be verified computationally, which has been done independently by Edward D.~Kim and Julian Pfeifle with the software~\url{polymake}~\cite{polymake}. Still, in sections~\ref{sec:prismatoid} and~\ref{sec:sphere-maps}
we give two computer-free (but not ``computation-free'') proofs. Before going into details we list some properties of $Q$ which follow directly from its definition:

\begin{itemize}
\item The first $24$ vertices (labeled $\mathit{1}^+$ to $\mathit{24}^+$) and the last $24$ vertices (labeled $\mathit{1}^-$ to $\mathit{24}^-$) span two facets of $Q$, that we denote $Q^+$ and $Q^-$, lying in the hyperplanes $\{x_5=+1\}$ and $\{x_5=-1\}$. Hence, $Q$ is indeed a prismatoid.
 
\item 
$Q$ is symmetric under the orthogonal transformation $(x_1,x_2,x_3,x_4,x_5) \mapsto (x_4,x_3,x_1,x_2, -x_5)$, and this symmetry sends $Q^+$ to $Q^-$, with the vertex labeled $\mathit{i}^+$ going to the one labeled $\mathit{i}^-$. (Observe, however, that this symmetry is not an involution).

\item $Q^+$ and $Q^-$ (hence also $Q$) are themselves invariant under any of the following 
 $32$ orthogonal transformations:
\[
\begin{pmatrix}
\pm1 & 0 & 0 & 0 & 0 \\
0 & \pm1 & 0 & 0 & 0 \\
0 & 0 & \pm1 & 0 & 0 \\
0 & 0 & 0 & \pm1 & 0 \\
0 & 0 & 0 & 0 & 1 \\
\end{pmatrix},
\qquad
\begin{pmatrix}
0 & \pm1 & 0 & 0 & 0 \\
\pm1 & 0 & 0 & 0 & 0 \\
0 & 0 & 0 & \pm1 & 0 \\
0 & 0 & \pm1 & 0 & 0 \\
0 & 0 & 0 & 0 & 1 \\
\end{pmatrix}.
\]
That is, they are symmetric under changing the sign of any of the first four coordinates and also under the simultaneous transpositions $x_1\leftrightarrow x_2$ and 
$x_3\leftrightarrow x_4$.
\end{itemize}

In what follows we denote $\Sigma$ the symmetry group of $Q$ (of order $64$) and $\Sigma^+$ the index-two subgroup that preserves $Q^+$ and $Q^-$.

\section{First proof of Theorem~\ref{thm:prismatoid}}
\label{sec:prismatoid}

The first proof of Theorem~\ref{thm:prismatoid} goes by explicitly describing the adjacency graph between orbits of facets of $Q$ which, thanks to symmetry, is not too difficult:


\begin{theorem}
\label{thm:facets}
\begin{enumerate}
\item The $2 + 20\times 16 = 322$ inequalities of Table~\ref{table:prsimatoid-facets} define facets of $Q$. $\rm A$ and $\rm L$ are the bases of the prismatoid. Among the rest, the $32$ labeled with the same letter form a $\Sigma^+$-orbit. There are six $\Sigma$-orbits, obtained as the $\Sigma^+$-orbit unions $\rm A \cup L$, $\rm B\cup K$, $\rm C\cup J$, $\rm D\cup I$, $\rm E\cup H$, and $\rm F\cup G$.
\begin{table}[htb]
\small
\[
\begin{array}{rrlcllll}
\hline
\rm A:  &1&-x_5 &\ge & 0.\medskip\cr
\rm B_{\pm,\pm,\pm,\pm}: &\frac{315}{2} &-\frac{135}{2} x_5 &\ge & \pm 5x_1 & \pm x_2 & \pm 2x_3 & \pm x_4.\medskip \cr
\rm B'_{\pm,\pm,\pm,\pm}:  &\frac{315}{2} &-\frac{135}{2} x_5 &\ge & \pm 5x_2 & \pm x_1 & \pm 2x_4 & \pm x_3.\cr\cr
\rm C_{\pm,\pm,\pm,\pm}:&135 &-45x_5 &\ge & \pm 4 x_1 & \pm 2 x_2 & \pm\frac{7}{4} x_3 & \pm \frac{5}{4} x_4.\medskip\cr
\rm C'_{\pm,\pm,\pm,\pm}:  &135 &-45x_5 &\ge & \pm4x_2 & \pm 2 x_1 & \pm\frac{7}{4} x_4 & \pm \frac{5}{4}x_3.\cr\cr
\rm D_{\pm,\pm,\pm,\pm}: &135 &-45x_5&\ge & \pm  4 x_1 & \pm x_2 & \pm 2 x_3 & \pm x_4.\medskip\cr
\rm D'_{\pm,\pm,\pm,\pm}:  &135 &-45x_5&\ge & \pm  4 x_2 & \pm x_1 & \pm 2 x_4 & \pm x_3.\cr\cr
\rm E_{\pm,\pm,\pm,\pm}: &105 &-30 x_5&\ge & \pm 3x_1 & \pm \frac{3}{2}x_2 & \pm \frac{3}{2}x_3 & \pm  x_4.\medskip\cr
\rm E'_{\pm,\pm,\pm,\pm}:  &105 &-30 x_5&\ge & \pm 3x_2 & \pm \frac{3}{2}x_1 & \pm \frac{3}{2}x_4 & \pm  x_3.\cr\cr
\rm F_{\pm,\pm,\pm,\pm}: &75 &-15x_5&\ge & \pm  2 x_1 & \pm x_2 & \pm x_3 & \pm x_4 \medskip\cr
\rm F'_{\pm,\pm,\pm,\pm}:   &75 &-15x_5&\ge & \pm  2 x_2 & \pm x_1 & \pm x_4 & \pm x_3 \cr\cr
\rm G_{\pm,\pm,\pm,\pm}: & 75 &+15x_5&\ge & \pm  2 x_4 & \pm x_3 & \pm x_1 & \pm x_2 \medskip\cr
\rm G'_{\pm,\pm,\pm,\pm}:& 75 &+15x_5&\ge & \pm  2 x_3 & \pm x_4 & \pm x_2 & \pm x_1 \cr \cr
\rm H_{\pm,\pm,\pm,\pm}: & 105 &+30 x_5&\ge & \pm 3x_4 & \pm \frac{3}{2}x_3 & \pm \frac{3}{2}x_1 & \pm  x_2. \medskip\cr
\rm H'_{\pm,\pm,\pm,\pm}:& 105 &+30 x_5&\ge & \pm 3x_3 & \pm \frac{3}{2}x_4 & \pm \frac{3}{2}x_2 & \pm  x_1. \cr\cr
\rm I_{\pm,\pm,\pm,\pm}:  &135 &+45x_5&\ge & \pm  4 x_4 & \pm x_3 & \pm 2 x_1 & \pm x_2. \medskip\cr
\rm I'_{\pm,\pm,\pm,\pm}: &135 &+45x_5&\ge & \pm  4 x_3 & \pm x_4 & \pm 2 x_2 & \pm x_1. \cr\cr
\rm J_{\pm,\pm,\pm,\pm}:  & 135 &+45x_5 &\ge & \pm4x_4 & \pm 2 x_3 & \pm\frac{7}{4} x_1 & \pm \frac{5}{4}x_2. \medskip\cr
\rm J'_{\pm,\pm,\pm,\pm}: & 135 &+45x_5 &\ge & \pm4x_3 & \pm 2 x_4 & \pm\frac{7}{4} x_2 & \pm \frac{5}{4}x_1. \cr\cr
\rm K_{\pm,\pm,\pm,\pm}: & \frac{315}{2} &+\frac{135}{2} x_5 &\ge & \pm 5x_4 & \pm x_3 & \pm 2x_1 & \pm x_2. \medskip\cr
\rm K'_{\pm,\pm,\pm,\pm}:& \frac{315}{2} &+\frac{135}{2} x_5 &\ge & \pm 5x_3 & \pm x_4 & \pm 2x_2 & \pm x_1. \cr\cr
\rm L:& 1 &+ x_5 &\ge & 0.\cr
\hline
\end{array}
\]
\caption{The $322$ facets of $Q$}
\label{table:prsimatoid-facets}
\end{table}

\item These are all the facets of $Q$.

\item The only adjacencies between facets of $Q$ in different $\Sigma^+$-orbits of facets of $Q$ are the ones shown in Figure~\ref{fig:prismatoid-graph}.
\end{enumerate}
\end{theorem}

\begin{figure}[htb]
\begin{center}
\includegraphics[scale=0.85]{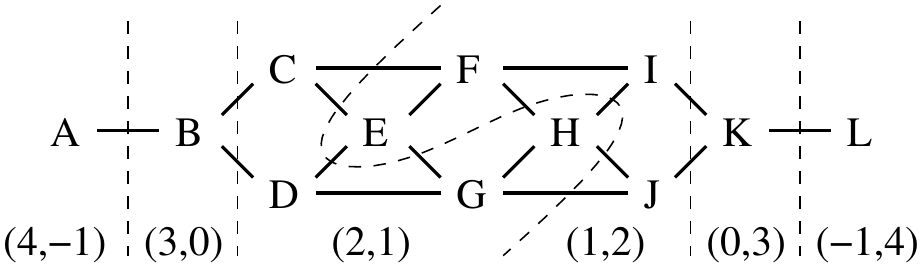}
\caption{The adjacencies between facets of $Q$, modulo symmetry. Dashed lines separate facets according to their bidimension. We say that a facet has bidimension $(i,j)$ if it is the convex hull of an $i$-face of $Q^+$ and a $j$-face of $Q^-$. The empty face is considered to have dimension $-1$
}
\label{fig:prismatoid-graph}
\end{center}
\end{figure}

Part 3 implies Theorem~\ref{thm:prismatoid}: in Figure~\ref{fig:prismatoid-graph} it is clear that six steps are necessary (and sufficient) to go from the facet $Q^+$  to the facet $Q^-$ (labeled $\rm A$ and $\rm L$).

\begin{proof}[Proof of Theorem~\ref{thm:facets}.]
In part one, the assertions about $\Sigma$ and $\Sigma^+$-orbits are straightforward, from the aspect of the inequalities. To check that these $322$ inequalities define facets, we need to consider a single representative from each $\Sigma$-orbit. For the base facets this is obvious, and for the rest we choose the following representatives:
\[
\begin{array}{rrlcllll}
\rm B_{+,+,+,+}: &\frac{315}{2} &-\frac{135}{2} x_5 &\ge &  5x_1 & + x_2 & + 2x_3 & + x_4.\medskip \cr
\rm C_{+,+,+,+}:&135 &-45x_5 &\ge & 4x_1 & + 2 x_2 & +\frac{7}{4} x_3 & + \frac{5}{4}x_4.\medskip\cr
\rm D_{+,+,+,+}: &135 &-45x_5&\ge &   4 x_1 & + x_2 & + 2 x_3 & + x_4.\medskip\cr
\rm E_{+,+,+,+}: &105 &-30 x_5&\ge &  3x_1 & + \frac{3}{2}x_2 & + \frac{3}{2}x_3 & +  x_4.\medskip\cr
\rm F_{+,+,+,+}: &75 &-15x_5&\ge &   2 x_1 & + x_2 & + x_3 & + x_4 \medskip\cr
\end{array}
\]

We leave it to the reader to check that these five inequalities are satisfied on all $48$ vertices of $Q$, and with equality precisely in the vertices listed for each in Table~\ref{table:vertex-facet}. This task is not as hard as it seems since only the vertices with non-negative coordinates $x_1$, $x_2$, $x_3$ and $x_4$ need to be checked (there are sixteen of them).
The five matrices have rank five, which shows that the vertices in each of them span at least an affine hyperplane. Hence, they all define facets of $Q$.

\begin{table}
\[
\begin{array}{cc}
\text{facet} & \text{vertices} \cr
\hline
\rm B_{+,+,+,+}: &
%
\bordermatrix{
&x_1&x_2&x_3&x_4&x_5\cr
\mathit{1}^+&  18&   0&   0&   0&   1 \cr
\mathit{5}^+&   0&   0&  45&   0&   1 \cr
\mathit{9}^+&  15&  15&   0&   0&   1 \cr 
\mathit{13}^+&   0&   0&  30&  30&   1 \cr
\mathit{17}^+&   0&  10&  40&   0&   1 \cr
\mathit{21}^+&  10&   0&   0&  40&   1 \cr
\mathit{5}^-&  45&   0&   0&   0&  -1
}\cr
%
\rm C_{+,+,+,+}:&
%
\bordermatrix{
&x_1&x_2&x_3&x_4&x_5\cr
\mathit{9}^+&  15&  15&   0&   0&   1 \cr 
\mathit{13}^+&   0&   0&  30&  30&   1 \cr
\mathit{17}^+&   0&  10&  40&   0&   1 \cr
\mathit{21}^+&  10&   0&   0&  40&   1 \cr
\mathit{5}^-&  45&   0&   0&   0&  -1 \cr
\mathit{13}^-&  30&  30&   0&   0&  -1 \cr
}\cr
%
\rm D_{+,+,+,+}: &
%
\bordermatrix{
&x_1&x_2&x_3&x_4&x_5\cr
\mathit{5}^+&   0&   0&   45&  0&   1 \cr
\mathit{13}^+&   0&   0&  30&  30&   1 \cr
\mathit{17}^+&   0&  10&  40&   0&   1 \cr
\mathit{5}^-&  45&   0&   0&   0&  -1 \cr
\mathit{17}^-&   40&  0&   10&  0&  -1 \cr
}\cr
%
\rm E_{+,+,+,+}: &
%
\bordermatrix{
&x_1&x_2&x_3&x_4&x_5\cr
\mathit{13}^+&   0&   0&  30&  30&   1 \cr
\mathit{17}^+&   0&  10&  40&   0&   1 \cr
\mathit{5}^-&  45&   0&   0&   0&  -1 \cr
\mathit{13}^-&  30&  30&   0&   0&  -1 \cr
\mathit{17}^-&   40&  0&   10&  0&  -1 \cr
}\cr
%
\rm F_{+,+,+,+}: &
%
\bordermatrix{
&x_1&x_2&x_3&x_4&x_5\cr
\mathit{13}^+&   0&   0&  30&  30&   1 \cr
\mathit{21}^+&  10&   0&   0&  40&   1 \cr
\mathit{5}^-&  45&   0&   0&   0&  -1 \cr
\mathit{13}^-&  30&  30&   0&   0&  -1 \cr
\mathit{17}^-&  40&   0&  10&   0&  -1 \cr
}\cr
\hline
\end{array}
\]
\caption{Vertex-facet incidence for the representative facets}
\label{table:vertex-facet}
\end{table}

For parts 2 and 3 we look more closely at the matrices in Table~\ref{table:vertex-facet} and observe that:
\begin{itemize}
\item The facets of types $D$, $E$ and $F$ are simplices, since they have five vertices: three in $Q^+$ and two in $Q^-$ or vice-versa.
\item The facets of type $C$ are iterated pyramids, with two apices in $Q^-$, over a quadrilateral in $Q^+$. Indeed, the vertices $\mathit{9}^+$, $\mathit{13}^+$, $\mathit{17}^+$ and 
$\mathit{21}^+$ form a planar quadrilateral, since
\[
2 v_{\mathit{9}^+} + 4 v_{\mathit{13}^+} = 3 v_{\mathit{17}^+} + 3 v_{\mathit{21}^+}.
\]
\item The facets of type $B$ are pyramids, with apex in $Q^-$ over a triangular prism in $Q^+$. Indeed, the six vertices in $Q^+$ form a triangular prism since the rays $\overrightarrow{v_{\mathit{1}^+} v_{\mathit{5}^+}}$, 
$\overrightarrow{v_{\mathit{9}^+} v_{\mathit{17}^+}}$,  and $\overrightarrow{v_{\mathit{21}^+} v_{\mathit{13}^+}}$ collide at the point  $o=(-30,0,120,0,1)$. This follows from the following equalities, and is illustrated in Figure~\ref{fig:prism-b}.
\[
\frac{8}{3}v_{\mathit{5}^+} - \frac{5}{3} v_{\mathit{1}^+} =
3v_{\mathit{17}^+} - 2v_{\mathit{9}^+}=
4 v_{\mathit{13}^+} - 3 v_{\mathit{21}^+}=
o.
\]
\end{itemize}
\begin{figure}[htb]
\begin{center}
\includegraphics[scale=0.7]{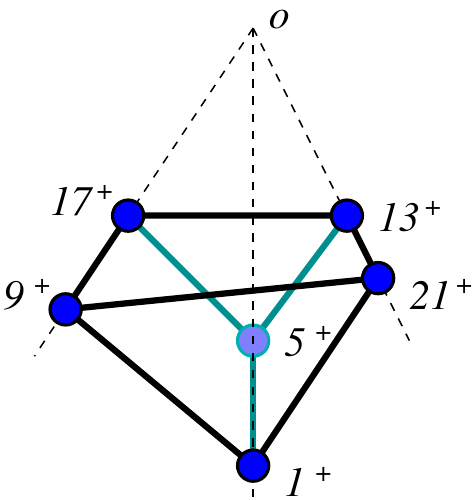}
\caption{The prism contained in $B_{+,+,+,+}$}
\label{fig:prism-b}
\end{center}
\end{figure}

With this information, we can prove parts 2 and 3 together by simply constructing the dual graph, that is, by looking at adjacencies between facets. The simplices of types $D$, $E$ and $F$ need to have five neighboring facets, the double pyramid of type $C$ needs six, and the pyramid over a triangular prism needs six too. So, it suffices to check the following adjecencies, which can be done by tracking the vertices in the five representative facets and the permutations of facets  within each $\Sigma$-orbit induced by the symmetries of $Q$:
\begin{itemize}
\item The following facets are neighbors of $B_{+,+,+,+}$:
\[
\begin{array}{cccccc}
A& 
B_{+,-,+,+} & 
B_{+,+,-,+} & 
B_{+,+,+,-} &
C_{+,+,+,+} & 
D_{+,+,+,+} \cr
\end{array}
\]
\item The following facets are neighbors of $C_{+,+,+,+}$:
\[
\begin{array}{cccccc}
B_{+,+,+,+}   & 
C_{+,+,-,+}&
C_{+,+,+,-} &
C'_{+,+,+,+} &
E_{+,+,+,+} & 
F_{+,+,+,+} \cr
\end{array}
\]
\item The following facets are neighbors of $D_{+,+,+,+}$:
\[
\begin{array}{ccccc}
B_{+,+,+,+} & 
D_{+,-,+,+}  & 
D_{+,+,+,-} &
E_{+,+,+,+}  &  
G_{+,+,+,+} \cr
\end{array}
\]
\item The following facets are neighbors of $E_{+,+,+,+}$:
\[
\begin{array}{ccccc}
C_{+,+,+,+} & 
D_{+,+,+,+} & 
E_{+,+,+,-} &  
F_{+,+,+,+} &  
G_{+,+,+,+} \cr
\end{array}
\]
\item The following facets are neighbors of $F_{+,+,+,+}$:
\[
\begin{array}{ccccc}
C_{+,+,+,+} &
E_{+,+,+,+} & 
F_{+,-,+,+} &
H'_{+,+,+,+} & 
I'_{+,+,+,+} \cr
\end{array}
\]
\end{itemize}
\end{proof}

\section{Second proof of Theorem~\ref{thm:prismatoid}}
\label{sec:sphere-maps}

The second proof of Theorem~\ref{thm:prismatoid} reduces the study of the combinatorics of $d$-prismatoids to that of \emph{pairs of geodesic maps} in the $(d-2)$-sphere.

\subsection{From prismatoids to pairs of geodesic maps}
\label{sec:maps}

The intersection of a prismatoid with an intermediate hyperplane equals
the Minkowski sum $Q^+ + Q^-$ of its two bases. More precisely:

\begin{proposition}
\label{prop:minkowski}
If $Q$ is a prismatoid with base facets $Q^+$ and $Q^-$ and $H$ is an intermediate hyperplane parallel to $Q^+$ and $Q^-$ then
\[
Q\cap H = \lambda_1 Q^+ + \lambda_2 Q^-,
\]
where $\lambda_1+\lambda_2=1$ and $\lambda_1: \lambda_2=\d(H,Q^+) : \d(H,Q^-)$ (the ratio of distances from $H$ to $Q^+$ and $Q^-$).
\qed
\end{proposition}

Every face $F$ (facet or not) of a Minkowski sum $Q^+ + Q^-$ decomposes uniquely as a sum $F^+ +F^-$ of faces of $Q^+$ and $Q^-$. We call bi-dimension  of $F$ the pair $(\dim(F^+),\dim(F^-))$. 
Proposition~\ref{prop:minkowski} implies that every facet of $Q$ other than the two bases induces a facet in the Minkowski sum $Q^+ + Q^-$. More precisely, the dual graph of $Q^+ +Q^-$ equals the dual graph of $Q$ with the two base facets removed. Since the facets of $Q^+ + Q^-$ corresponding to facets of $Q$ adjacent to $Q^+$ (respectively, to  $Q^-$) are those of bi-dimension $(d-1,*)$ (respectively,  $(*,d-1)$), the $d$-step property for $Q$ translates to the following property for the pair of polytopes $(Q^+,Q^-)$:

\begin{definition}
\label{def:dstep-minkowski}
Let $Q^+$ and $Q^-$ be two polytopes of dimension $d-1$ with the same or parallel affine spans.
The pair $(Q^+, Q^-)$  \emph{has the $d$-step property} if there is a sequence $F_1,F_2,\dots,F_{k}$ of facets of $Q^+ + Q^-$, with $k\le d-1$, such that:
\begin{itemize}
\item The bi-dimension of $F_1$ and the bi-dimension of $F_{k}$ are $(d-1,*)$ and $(*,d-1)$.
\item $F_i$ is adjacent to $F_{i+1}$ for all $i$.
\end{itemize}
\end{definition}

That is to say, we ask whether we can go from an $F$ to an $F'$ in $d-2$ steps, where $F$ and $F'$ are facets of $Q^+ + Q^-$ of a special type.

\begin{proposition}
\label{prop:antiprisms-to-sums}
A prismatoid with base facets $Q^+$ and $Q^-$ has the $d$-step property if, and only if, the pair $(Q^+, Q^-)$ has the $d$-step property.
\qed
\end{proposition}

The good thing about Proposition~\ref{prop:antiprisms-to-sums} is that it reduces the dimension by one. In order to study our
prismatoid of dimension five we only need to understand its two bases, of dimension four. Before doing this let us make 
one more translation, to the language of \emph{normal fans} or \emph{normal maps}.

The normal cone of a face $F$ of a polytope $P\subset\real^d$ is the set of linear functionals
\[
\{\phi\in(\real^d)^* : \phi|_F \text{ is constant  and } \max_{p\in P} \phi(p) = \max_{p\in F} \phi(p)\}.
\]
The normal cones of faces of $P$ form a complex $\normal(P)$ of polyhedral cones called the \emph{normal fan} of  $P$. 
We call the intersection of this fan with the unit sphere the \emph{normal map} of $P$. (This is called the \emph{gaussian map} of $P$ in~\cite{Fogel-Halperin-Weibel}. Incidentally, the reading of~\cite{Fogel-Halperin-Weibel} was our initial inspiration for attempting to disprove the Hirsch conjecture via prismatoids).
The normal map of a $d$-polytope $P$ is a polyhedral complex of  spherical polytopes  decomposing $S^{d-1}$. We call such objects \emph{geodesic maps}. (A priori, a geodesic map may not be the normal map of any polytope).

The following definition and theorem translate Proposition~\ref{prop:antiprisms-to-sums} into the language of geodesic maps. Note that since our polytopes $Q^+$ and $Q^-$ are meant to be facets of a $d$-polytope, their normal maps lie in the sphere $S^{d-2}$.

\begin{definition}
\label{defi:gauss}
Let $\gauss^+$ and $\gauss^-$ be two geodesic maps in the sphere $S^{d-2}$. 
\begin{enumerate}
\item The \emph{common refinement} of $\gauss^+$ and $\gauss^-$ is the geodesic map whose cells are all the possible intersections of a cell of $\gauss^+$ and a cell of $\gauss^-$.
\item The pair $(\gauss^+, \gauss^-)$ has the \emph{$d$-step property} if the 1-skeleton of their common refinement contains a path of length at most
$d-2$ from a vertex of $\gauss^+$ to a vertex of $\gauss^-$.
\end{enumerate}
\end{definition}

\begin{theorem}
\label{label:gauss}
Let $Q^+$ and $Q^-$ be two polytopes in $\real^{d-1}$, with normal maps $\gauss^+$ and $\gauss^-$.
The pair $(Q^+,Q^-)$ has the $d$-step property if and only if the pair $(\gauss^+, \gauss^- )$ has the $d$-step property.
\end{theorem}


Using the  formalism of geodesic maps, Santos, Stephen and Thomas~\cite{4prismatoids} have shown that $4$-prismatoids have the $d$-step property. That is to say, they show that if a pair of graphs is embedded in the $2$-sphere it is always possible to go from a vertex of one to a vertex of the other traversing at most two edges of their common refinement. The following example, which can be understood as a pair of geodesic maps in the flat torus, shows that the reason for this is not ``local''.

\begin{example}[A pair of periodic maps in the plane without the $d$-step property]
\label{exm:plane-gauss}
\rm
Figure~\ref{fig:2d-map} shows (portions of) two maps drawn in the plane, one with its 1-skeleton in black and the other in grey. They are meant to be periodic and have 
the same symmetry group, vertex-transitive in both. The pair does not have the $d$-step property. We cannot go in two steps from a vertex $v$ of the black map to a vertex of the grey map.
\begin{figure}[htb]
\begin{center}
\includegraphics[scale=.8]{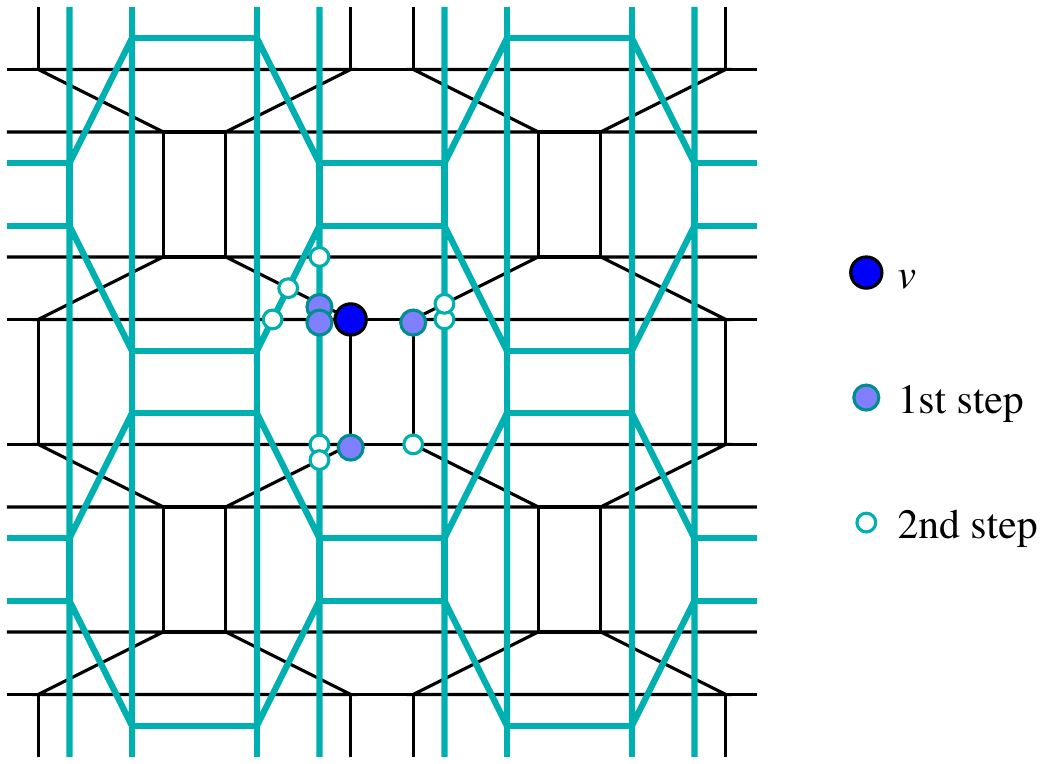}
\caption{A pair of periodic maps in the plane without the $d$-step property}
\label{fig:2d-map}
\end{center}
\end{figure}
%
%
%
\end{example}

\subsection{A pair of geodesic maps in $S^3$ without the $d$-step property}
\label{sub:maps}





Let us now concentrate on the facet $Q^+$ of our $5$-dimensional prismatoid $Q$. 
In the light of Proposition~\ref{prop:antiprisms-to-sums} we can think of $Q^+$ as lying in $\real^4$, and omit its last coordinate.
%
%

\begin{lemma}
\label{lemma:Q^+-facets}
$Q^+$ has 32 facets, given by the following inequalities:
\begin{eqnarray*}
\pm 5x_1 \pm x_2 \pm 2x_3 \pm x_4 \le 90, \\
\pm x_1 \pm 5x_2 \pm x_3 \pm 2x_4 \le 90.
\end{eqnarray*}
The symmetry group $\Sigma^+$ of $Q^+$ acts transitively on them.
\end{lemma}

Observe that Lemma~\ref{lemma:Q^+-facets} describes $Q^+$ as the intersection of two cross-polytopes, so its polar is the common convex hull of two combinatorial $4$-cubes. This partially explains why this polar is a \emph{cubical polytope} (see Remark~\ref{rem:cubical} below).

\begin{proof}
That $\Sigma^+$ acts transitively on these inequalities is clear from the  form of them and the description of $\Sigma^+$ in Section~\ref{sec:symmetry}. Symmetry has the consequence that in order to prove that all the inequalities define facets we just need to consider one of them.
Take for instance:
\[
5x_1 + x_2 + 2x_3 + x_4 \le 90.
\]
Direct inspection shows that the inequality is valid on the $24$ vertices of $Q^+$, and that it is met with equality precisely in the following six:
\[
\bordermatrix{
&x_1&x_2&x_3&x_4\cr
\mathit{1}^+&  18&   0&   0&   0 \cr
\mathit{5}^+&   0&   0&  45&   0    \cr
\mathit{9}^+&  15&  15&   0&   0   \cr 
\mathit{13}^+&   0&   0&  30&  30    \cr
\mathit{17}^+&   0&  10&  40&   0    \cr
\mathit{21}^+&  10&   0&   0&  40    \cr
}
\]
Since the top $4\times 4$ submatrix is regular, these six points span the affine hyperplane $5x_1 + x_2 + 2x_3 + x_4 = 90$, hence they define a facet $F$.

We still need to check that there are no other facets apart from the $32$ in the statement. For this consider the following equalities, in which $v_{i^+}$ represents (the actual vector of coordinates of) the vertex labeled ${i}^+$:
\[
\frac{8}{3}v_{\mathit{5}^+} - \frac{5}{3} v_{\mathit{1}^+} =
3v_{\mathit{17}^+} - 2v_{\mathit{9}^+}=
4 v_{\mathit{13}^+} - 3 v_{\mathit{21}^+}=
o.
\]
These equalities say that the rays $\overrightarrow{v_{\mathit{1}^+} v_{\mathit{5}^+}}$, 
$\overrightarrow{v_{\mathit{9}^+} v_{\mathit{17}^+}}$,  and $\overrightarrow{v_{\mathit{21}^+} v_{\mathit{13}^+}}$ collide at the point  $o=(-30,0,120,0,1)$,
so that $F$ is combinatorially a triangular prism, as was shown
in Figure~\ref{fig:prism-b}.
%
%
So, we only need to check that the five neighbors of $F$ are in our stated list of facets. This is true since:
\begin{itemize}
\item $v_{\mathit{5}^+}, v_{\mathit{13}^+}, v_{\mathit{17}^+} \in \{-5x_1 + x_2 + 2x_3 + x_4 = 90\}$ ($x_1=0$ in these points).
\item $v_{\mathit{1}^+}, v_{\mathit{5}^+}, v_{\mathit{13}^+} v_{\mathit{21}^+} \in \{5x_1 - x_2 + 2x_3 + x_4 = 90\}$ ($x_2=0$ in them).
\item $v_{\mathit{1}^+}, v_{\mathit{9}^+}, v_{\mathit{21}^+} \in \{5x_1 + x_2 - 2x_3 + x_4 = 90\}$ ($x_3=0$ in them).
\item $v_{\mathit{1}^+}, v_{\mathit{5}^+}, v_{\mathit{9}^+}, v_{\mathit{17}^+} \in \{5x_1 +x_2 + 2x_3 - x_4 = 90\}$ ($x_4=0$ in them).
\item $v_{\mathit{9}^+}, v_{\mathit{13}^+}, v_{\mathit{17}^+}, v_{\mathit{21}^+} \in \{x_1 + 5x_2 + x_3 + 2x_4 = 90\}$.
\end{itemize}
\end{proof}

This description of the facets of $Q^+$ translates nicely to the normal map $\gauss^+$ of $Q^+$. For the sake of having nice integer coordinates, we consider this  map as lying in the sphere of radius $\sqrt{31}$ rather than radius $1$ (but we still denote  this dilated sphere $S^3$, for simplicity). That is, the vertices of $\gauss^+$ are the following 
points:
\[
p_{\pm,\pm,\pm,\pm}:=(\pm 5, \pm 1, \pm 2, \pm 1), \qquad \quad
p'_{\pm,\pm,\pm,\pm}:=(\pm 1, \pm 5, \pm 1, \pm 2).
\]
Observe they all lie in a torus $\{x_1^2 + x_2^2 = 26, x_3^2 + x_4^2 = 5\}$, which we denote $T^+$
It is quite natural then to picture them on the flat torus. This is what we do in Figure~\ref{fig:G_1}, where the horizontal and vertical coordinates represent the angle along the circles $\{x_1^2 + x_2^2 = 26\}$ and $\{x_3^2 + x_4^2 = 5\}$, respectively. 
\begin{figure}[htb]
\begin{center}
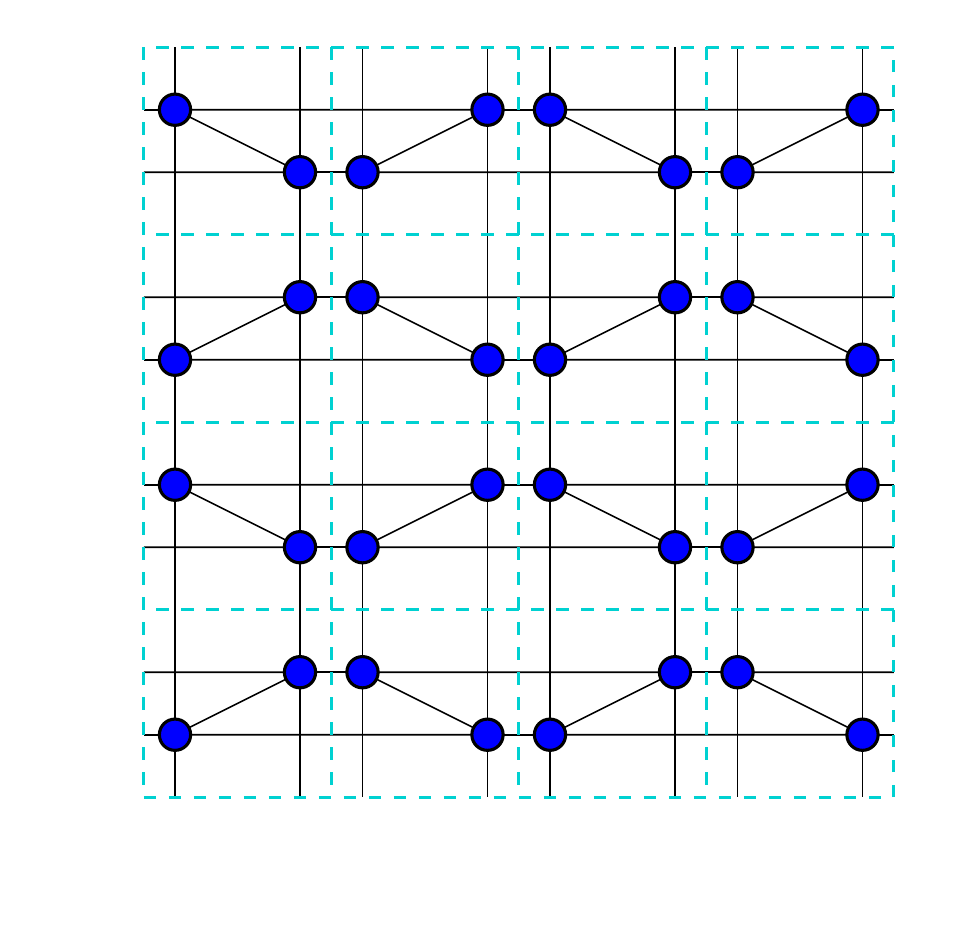
\caption{$\gauss^+$, drawn on a flat torus}
\label{fig:G_1}
\end{center}
\end{figure}
The grey dashed lines divide the torus into its sixteen octants. The 32 dots are the vertices of $\gauss^+$, and the segments joining them represent its edges. 
Of course, the true edges in $S^3$ do not go along $T^+$. In particular, the crossings we see in the picture are an artifact.

As seen in the matrix defining $Q^+$, there are five orbits of facets of $\gauss^+$ (i.e.,~vertices of $Q^+$) modulo $\Sigma^+$. Representatives of them are, for example, the normal cones of vertices $\mathit{3}^+$, $\mathit{7}^+$, $\mathit{9}^+$, $\mathit{13}^+$, and $\mathit{22}^+$ which we highlight 
in Figure~\ref{fig:G_1-facets}. The eight ``vertical strips'' in the orbits of $\cone(\mathit{3}^+)$ and $\cone(\mathit{9}^+)$ glue together to form a (polyhedral) solid torus subdivided into eight slices, each of which is combinatorially a $3$-cube. The eight horizontal strips form a second solid torus. These two tori are glued along the sixteen diagonal edges and eight of the vertical rectangles in the pictures, leaving eight empty regions in between. These regions are filled in by the other eight cones, in the orbit of $\cone(\mathit{22}^+)$.

\begin{remark}
\label{rem:cubical}
\rm
All the facets of $\gauss^+$ are combinatorially equivalent to the  $3$-cube. That is, $Q^+$ is polar to a \emph{cubical polytope}. In fact, it was communicated to us by M.~Joswig and G.~Ziegler that the polar of $Q^+$ is one of the \emph{neighborly cubical polytopes} (with the graph of the $5$-cube) that they constructed in~\cite{JoswigZiegler} and had been previously found by Blind and Blind~\cite{BlindBlind}.
\end{remark}

\begin{figure}[htb]
\begin{center}
\begin{tabular}{cc}
\includegraphics[scale=0.45]{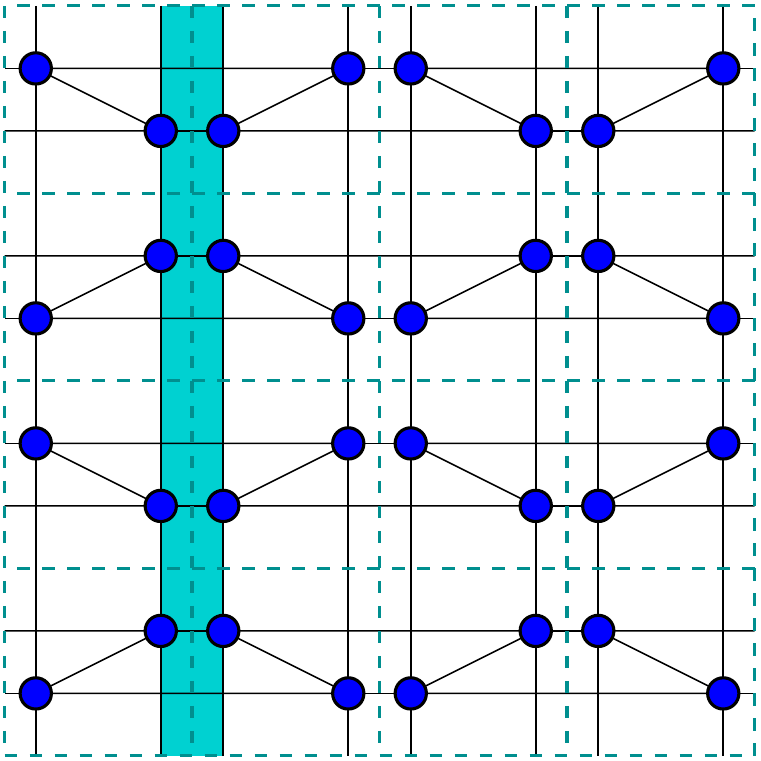}&
\includegraphics[scale=0.45]{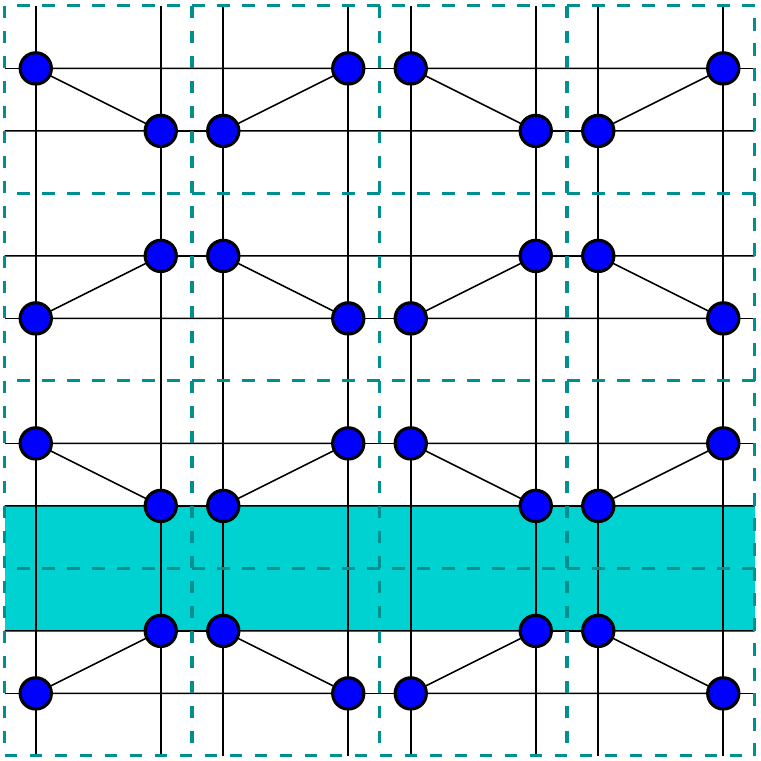}\\
$\cone(\mathit{3}^+)$ &$\cone(\mathit{7}^+)$ \\
\end{tabular}
\begin{tabular}{ccc}
\includegraphics[scale=0.45]{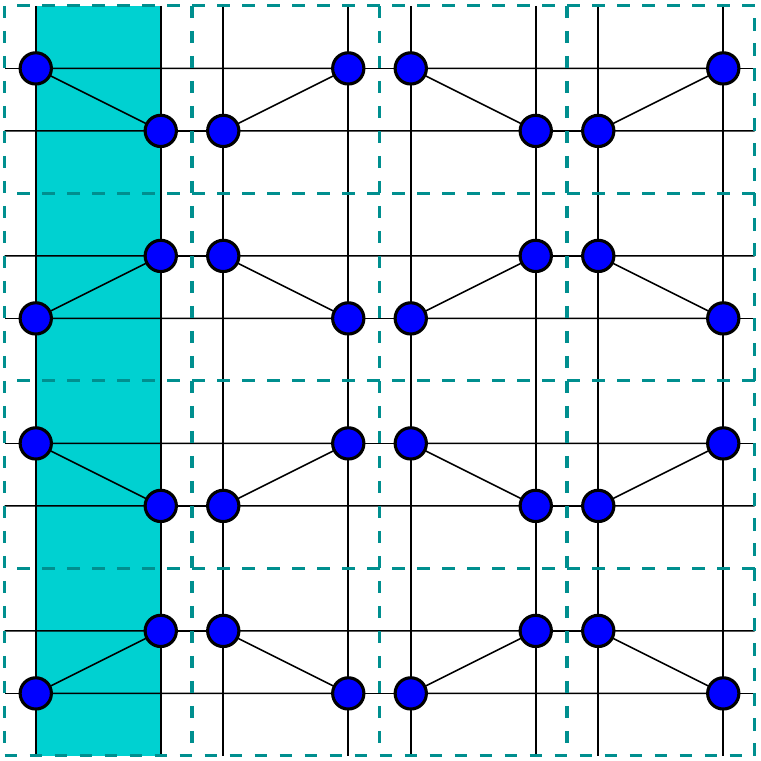}&
\includegraphics[scale=0.45]{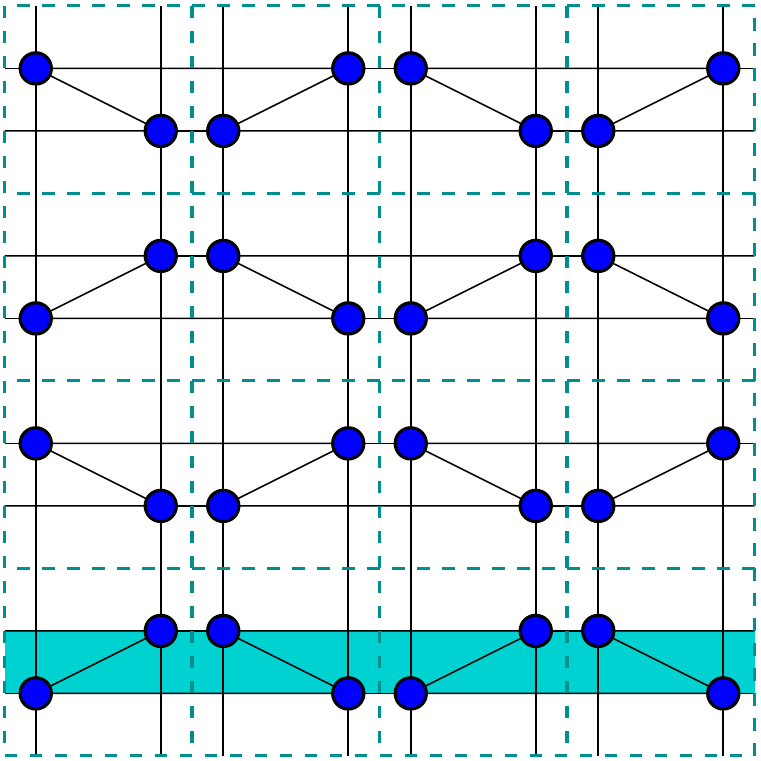}&
\includegraphics[scale=0.45]{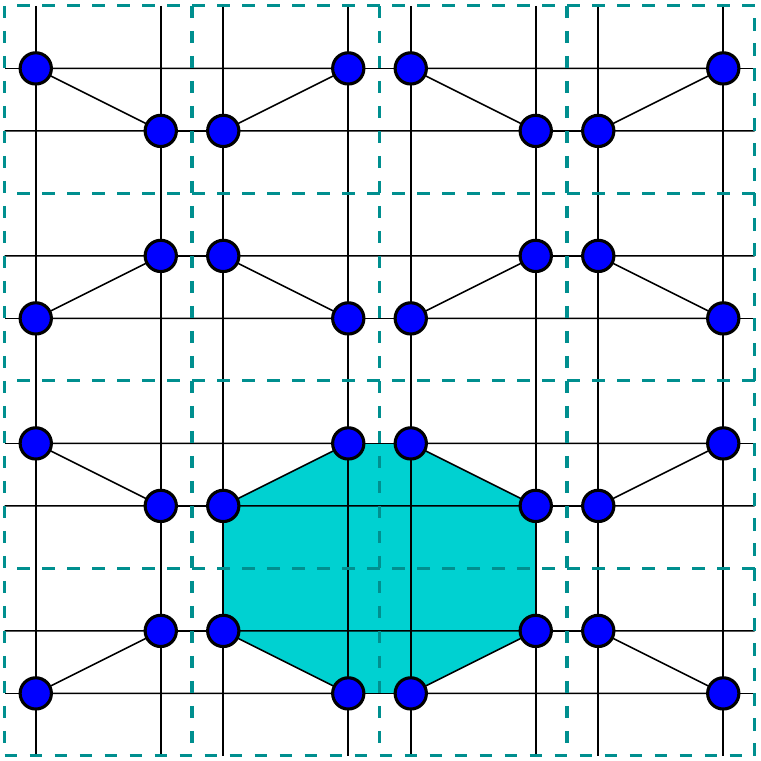}\\
$\cone(\mathit{9}^+)$ &$\cone(\mathit{13}^+)$ &$\cone(\mathit{22}^+)$ \\
\end{tabular}
\caption{Representatives of the five orbits of facets of $\gauss^+$.}
\label{fig:G_1-facets}
\end{center}
\end{figure}

Now that we understand $\gauss^+$, let us 
draw the normal map $\gauss^-$ of $Q^-$, whose vertices are (remember the symmetry that sends $Q^+$ to $Q^-$): 
\[
(\pm 1, \pm 2, \pm 5, \pm 1), \qquad \quad
(\pm 2, \pm 1, \pm 1, \pm 5).
\]
Although these points lie in a \emph{different} torus in $S^3$ than the vertices of $\gauss^+$, it is natural to draw them in the same flat torus (Figure~\ref{fig:Q}). Basically, what we are doing is projecting every point $(a e^{ix},be^{iy})$ of the sphere $S^3\subset \C^2$, to its longitude and latitude $(x,y)$ in the square. 
Observe the similarity betwen Figures~\ref{fig:Q} and~\ref{fig:2d-map}.
\begin{figure}[htb]
\begin{center}
\includegraphics[scale=0.8]{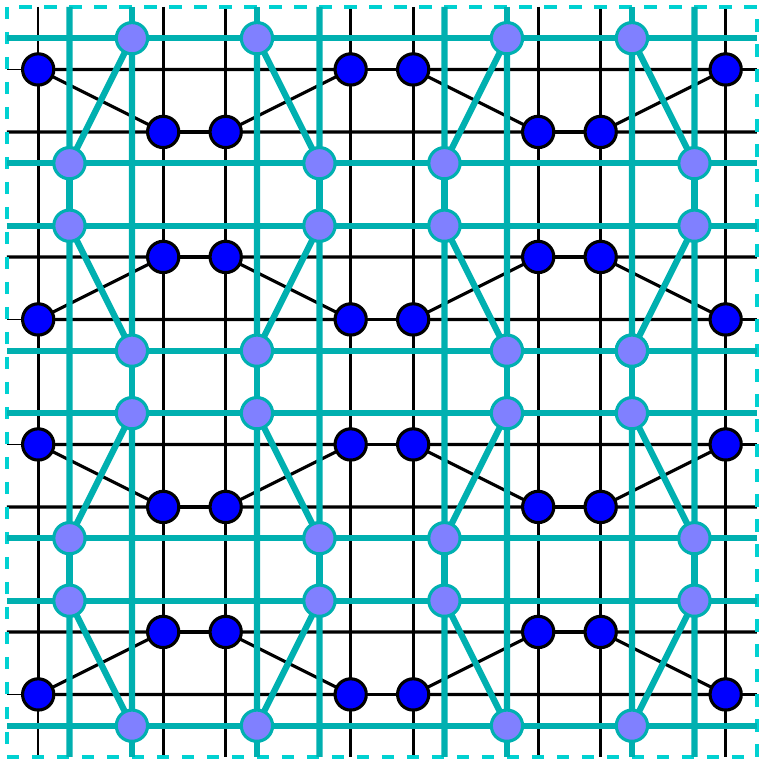}
\caption{The normal maps $\gauss^+$ (dark) and $\gauss^-$ (light), drawn on the same flat torus}
\label{fig:Q}
\end{center}
\end{figure}
%


\begin{lemma}
\label{lemma:vertex-in-cell}
\begin{enumerate}
\item Every vertex of $\gauss^-$ lies in the interior of one of the four facets of $\gauss^+$ of the orbit of $\cone(\mathit{7}^+)$.
\item Similarly, every vertex of $\gauss^+$ lies in the interior of one of the four facets of $\gauss^-$ of the orbit of $\cone(\mathit{7}^-)$.
\item If $v$ is a vertex of $\gauss^+$ and $C$ the facet of $\gauss^-$ containing it, no vertex of $C$ lies in a facet of $\gauss^+$ having $v$ as a vertex.
\end{enumerate}
\end{lemma}

\begin{proof}
Since both maps have the same symmetry group (the group $\Sigma^+$ of the previous section) and the group is transitive on the vertices, we only need to prove the lemma for a single vertex $v$. We take $v=(5,1,2,1)$ and show that it is contained in the interior of the facet $\cone(\mathit{5}^-)$. By definition, 
the vertices of $\gauss^-$ on this facet are those satisfying the equation $45 x_1 = 90$, that is, $x_1=2$. They are the eight vertices of the form:
\[
(2, \pm1, \pm 1, \pm 5).
\]
Hence, the facet inequality description of $\cone(\mathit{5}^-)$ is:
\[
\cone(\mathit{5}^-) = \{(p_1,p_2,p_3,p_4)\in S^3 : \pm 2 p_2 \le p_1, \pm 2 p_3 \le  p_1, \pm 2 p_4 \le 5 p_1\}.
\]
The proof of part one (and two) finishes by noticing that $(5,1,2,1)$ satisfies these six inequalities strictly.

For part three, we look at what facets of $\gauss^+$ contain the vertices of $\cone(\mathit{5}^-)$. They have to be in the $\Sigma^+$-orbit of  $\cone(\mathit{7}^+)$ 
and the pictures (see Figure~\ref{fig:conclusion}) tell us that they are the facets
$\cone(\mathit{7}^+)$  and  $\cone(\mathit{8}^+)$, none of whose vertices is the original $v$.
\end{proof}
\begin{figure}[htb]
\begin{center}
\includegraphics[scale=0.6]{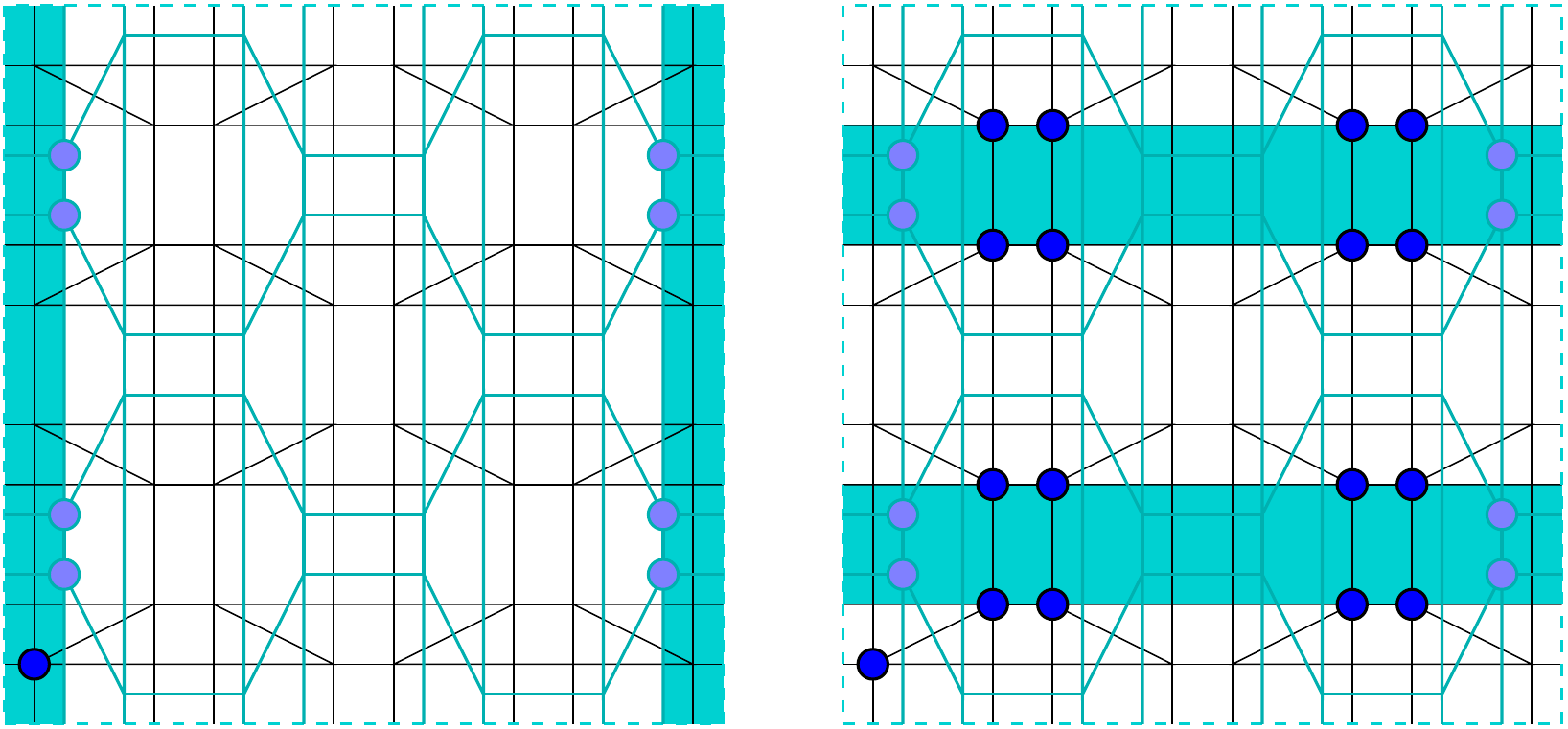}
\caption{Proof of Lemma~\ref{lemma:vertex-in-cell}. The facet $C$ of $\gauss^+$ containing $v_1$ (left) and the two facets of $\gauss^-$ containing the vertices of $C$ (right)}
\label{fig:conclusion}
\end{center}
\end{figure}

Part 3 of Lemma~\ref{lemma:vertex-in-cell} is the key to finishing the proof of Theorem~\ref{thm:prismatoid}. In the following statement we say that a pair $(\gauss^+, \gauss^-)$ of geodesic maps on $S^{d-2}$ are \emph{transversal} if whenever respective cells $C_1$ and $C_2$ of them intersect we have
\[
\dim(C_1) + \dim(C_2) = d-2 + \dim(C_1\cap C_2).
\]

\begin{proposition}
\label{prop:transversal}
Let $(\gauss^+, \gauss^-)$ be a transversal pair of geodesic maps in the $(d-2)$-sphere. If there is a path of length $d-2$ between a vertex $v_1$ of $\gauss^+$ and a vertex $v_2$ of $\gauss^-$, then the facet of $\gauss^-$ containing $v_1$ in its interior has $v_2$ as a vertex and the  facet of $\gauss^+$ containing $v_2$ in its interior has $v_1$ as a vertex.
\end{proposition}

\begin{proof}
For every cell $C_1\cap C_2$ of the common refinement of $\gauss^+$ and $\gauss^-$ we call $(\dim(C_1), \dim(C_2))$ the bi-dimension of $C_1\cap C_2$.
Let $v=C_1\cap C_2$ be a vertex incident to an edge $e=D_1\cap D_2$. Transversality implies that if the bi-dimension of $v$ is $(i,j)$, then the bi-dimension of $e$ is one of 
$(i+1,j)$ or $(i,j+1)$. (The former occurs if $C_1 \le D_1$ and $C_2=D_2$, and the latter if $C_1 = D_1$ and $C_2\le D_2$). As a consequence,
the bi-dimension of the other end of $e$ is one of $(i+1,j-1)$, $(i,j)$ or $(i-1,j+1)$. That is, bi-dimensions of consecutive vertices on a path differ by at most one unit on each coordinate.

Since $v_1$ and $v_2$ have bi-dimensions $(0,d-2)$ and $(d-2,0)$, along a path of $d-2$ steps connecting them
the first coordinate of the bi-dimension always increases and the second coordinate always decreases. This means that we move along a flag (a chain of cells each contained in the next) of $\gauss^+$, and at the end we finish in a facet having $v_1$ as a vertex. Similarly, the facet of $\gauss^-$ where we started has $v_2$ as a vertex.
\end{proof}

So, Theorem~\ref{thm:prismatoid} follows from Lemma~\ref{lemma:vertex-in-cell} and Proposition~\ref{prop:transversal} if we show that the pair 
 $(\gauss^+, \gauss^-)$ we are dealing with is transversal. The pair being transversal is equivalent to every proper face $F$ of $Q$ satisfying
 \[
 \dim(F\cap Q^+) + \dim(F\cap Q^-) = \dim(F)-1.
 \]
 This needs only be checked for facets, and was actually implicitly shown in the description of the facets of $Q$ given in Section~\ref{sec:prismatoid} (see Figure~\ref{fig:prismatoid-graph}, and the proof of parts 2 and 3 of Theorem~\ref{thm:facets}).
Alternatively, this second proof can be finished with a perturbation argument: Even if $(\gauss^+, \gauss^-)$ was not transversal, any sufficiently generic and sufficiently small rotation of one of the maps will make the pair transversal without destroying the property stated in part 3 of Lemma~\ref{lemma:vertex-in-cell}. 
 
 %

\section{An infinite family of non-Hirsch polytopes}
\label{sec:asymptotic}

In this section we show general procedures to construct new non-Hirsch polytopes from old ones. All the techniques we use are quite standard. Our result is that there is a fixed dimension $d$ in which we can build an infinite sequence of non-Hirsch polytopes with diameter exceeding the Hirsch bound by a fixed fraction. 

If we do not ask for a fixed dimension the result is straightforward (see, e.g.,~\cite[Prop.~1.3]{Klee:d-step}):

\begin{lemma}
\label{lemma:product}
If $P$ is a a $d$-polytope with $n$ facets and with diameter $(1+\epsilon)(n-d)$ for a certain $\epsilon>0$ then 
the $k$-fold product $P^k$ is a $kd$-polytope with $kn$ facets and with diameter $(1+\epsilon)(kn-kd)$. \qed
\end{lemma}

In fixed dimension we use the following glueing lemma. It appears, for example, in~\cite{Holt:Many-polytopes}:

\begin{lemma}
\label{lemma:glue}
Let $P_1$ and $P_2$  be simple polytopes  of the same dimension $d$,  having respectively $n_1$ and $n_2$ facets, and with diameters $l_1$ and $l_2$. Then, there is a  simple $d$-polytope with $n_1+n_2-d$ facets and with diameter at least $l_1+l_2-1$.
\qed
\end{lemma}


\begin{corollary}
\label{coro:glue}
If $P$ is a simple $d$-polytope with $n$ facets and diameter $l$, then for each $k$ there is a $d$-polytope $P_k$ with $k(n-d) + d$ facets and  diameter at least $k(l-1)+1$.
In particular, if $P$ is non-Hirsch then $P_k$ is non-Hirsch as well.
\end{corollary}

\begin{proof}
By induction on $k$, applying Lemma~\ref{lemma:glue} to $P_{k-1}$ and $P_1:=P$.
For the second part, assume that $P$ is non-Hirsch. That is, $l\ge n-d+1$. Then:
\[
k(n-d) +d -d = k(n-d) \le k(l-1),
\]
so $P_k$ is non-Hirsch as well.
\end{proof}

\begin{remark}[McMullen, personal communication]
\label{rem:mcmullen}
\rm
It is a consequence of Corollary~\ref{coro:glue} that if there is a linear bound, say $H(n,d)\le an+b$, for the diameters of polytopes of a fixed dimension $d$, then one has also the bound $H(n,d)\le a(n-d)+1$ (in the same dimension). Indeed, from a $d$-polytope with $n$ facets and diameter $a(n-d)+2$ (or higher), the corollary above gives $d$-polytopes in which the ratio of diameter to number of facets tends to (at least)
\[
\lim_{k\to \infty}\frac{k(a(n-d)+1)+1}{k(n-d) + d} =  a +\frac{1}{n-d} > a. 
\]
%
%
As one referee pointed out to us, this remark is reminiscent of Lemma 2 in~\cite{Barnette-LBT}: let $F_1(n,d)$ denote the maximum number of edges of all simplicial $d$-polytopes with $n$ vertices. If  $F_1(n,d) \ge dn-b$ holds all $n$ and $d$ and some constant $b$, then  $F_1(n,d) \ge d(n-d) + {d \choose 2}$ also holds.
\end{remark}

It is a consequence of Corollary~\ref{coro:glue} (and a special case of the remark above) that the dimensions for which Theorem~\ref{thm:asymptotic} holds are those for which there is a polytope violating the Hirsch bound by \emph{at least two}.  Lemma~\ref{lemma:product} implies that this is the case for dimension $2d$ if there is a non-Hirsch $d$-polytope. To give a more explicit statement, we call \emph{Hirsch excess} (or simply \emph{excess}) of a $d$-polytope with $n$ facets and diameter $l$ the ratio
\[
\frac{l}{n-d} -1.
\]

\begin{theorem}
\label{thm:asymptotic2}
Let $P$ be a non-Hirsch polytope of dimension $d$ and excess $\epsilon$. Then, for each $k\in\N$ there is an infinite family of non-Hirsch polytopes of dimension $kd$ and with excess greater than
\[
\left(1-\frac{1}{k}\right)\epsilon.
\]
\end{theorem}

\begin{proof}

Let $n$ be the number of facets of $P$ and let $l=(n-d)(\epsilon +1)$ be its diameter. The $k$-fold power $P^k$ has dimension $kd$, it has $kn$ facets and it has diameter $kl$. Now, glue an arbitrary number, say $j$ of copies of $P^k$ to one another.
Corollary~\ref{coro:glue} says that the polytope $P_{k,l}$ so obtained has dimension $kd$, it has $j(kn-kd)+kd$ facets and it has diameter $j(kl-1)+1$. Let us compute its excess:
\[
\frac{j(kl-1)+1}{j(kn-kd)}-1=\frac{jk(l-n+d)-j+1}{jk(n-d)} =\epsilon- \frac{j-1}{jk(n-d)} > \epsilon- \frac{1}{k(n-d)}.
\]
To finish the proof we just need to show that $\frac{1}{n-d}\le \epsilon$. This is so because $\epsilon=\frac{l-n+d}{n-d}$, and $l-n+d\ge 1$.
\end{proof}

In particular, starting with the non-Hirsch polytope of Corollary~\ref{coro:nonHirsch}, which has dimension $43$, $86$ facets and diameter $44$ ($\epsilon=1/43$), we can get infinite sequences of excess $1/86$ in dimension $86$ and of excess as close to $1/43$ as we want in fixed (but very high) dimension $d$. With the improved counterexamples announced in~\cite{5prismatoids} the numbers $43$ and $86$ can be replaced by $20$ and $40$, respectively.

\begin{remark}
\label{rem:excess}
\rm
From the last sentence in the proof of Theorem~\ref{thm:asymptotic2} we see that if $P$ violates the Hirsch inequality by an amount $b=l-n+d$ greater than $1$ then the conclusion of the theorem can be improved to 
\[
\left(1-\frac{1}{bk}\right)\epsilon.
\]
That is, the lower bound for the excess that we get in each dimension is slightly increased, but it is still smaller than the original excess $\epsilon$ of $P$. We do not know of any operation that can be applied to a non-Hirsch polytope and yield another one with higher excess.
%
\end{remark}

\subsection*{Acknowledgements}
I started seriously looking at the Hirsch Conjecture during my sabbatical stay at UC Davis in 2007-08. I thank UC Davis and the Spanish Ministry of Science for supporting my stay, and Jes\'us A.~De Loera and Edward D.~Kim for pushing me into the topic.
I also thank Jes\'us, Eddie and Julian Pfeifle for the computational verification of Theorem~\ref{thm:prismatoid}.

Besides the three above the following people sent me useful comments on previous versions of this paper: Michele Barbato, Louis J.~Billera, Enrique Castillo, Fred B.~Holt, Michael Joswig, Gil Kalai, Peter Kleinschmidt, Luis M.~Pardo, Vincent Pilaud, Tom\'as Recio, Bernd Sturmfels, Michael J.~Todd and G\"unter M.~Ziegler. I  also thank the two anonymous referees for several corrections and suggestions.


\vskip.5cm
\noindent {\small Francisco Santos}\newline
\emph{Departamento de Matem\'aticas, Estad\'istica y Computaci\'on}\newline
\emph{Universidad de Cantabria, E-39005 Santander, Spain}\newline
\emph{email: }\url{francisco.santos@unican.es}\newline
\emph{web: }\url{http://personales.unican.es/santosf/}


\end{document}